\documentclass[final,leqno]{siamltex}

\usepackage[cmex10]{amsmath}
\usepackage{amssymb} 

\usepackage{algorithmic}

\usepackage{array}

\usepackage{url}

\DeclareMathOperator*{\argmin}{arg\,min}
\DeclareMathOperator*{\mymax}{max}

\def\mplus{\mathrel{%
  \ooalign{\raise.29ex\hbox{$\scriptscriptstyle\mathbf{+}$}\cr}}}
\usepackage{color}

\usepackage{comment}
\newtheorem{remark}[theorem]{remark} 
\newtheorem{assumption}[theorem]{assumption} 

\usepackage{graphicx}
\usepackage{caption}
\usepackage{subcaption}
 \usepackage{float}

\begin{document}
\title{A Second-order Method for Compressed Sensing Problems with Coherent and Redundant Dictionaries \\ \vspace{0.5cm}\textnormal{M\MakeLowercase{ay} 29, 2014}}

\author{Ioannis~Dassios\thanks{I. Dassios is with the School of Mathematics and Maxwell Institute, The University of Edinburgh, Edinburgh,
Mayfield Road, Edinburgh EH9 3JZ, United Kingdom e-mail: idassios@ed.ac.uk. I. Dassios is supported by EPSRC Grant EP/I017127/1} \and
        Kimon~Fountoulakis\thanks{K. Fountoulakis is with the School of Mathematics and Maxwell Institute, The University of Edinburgh, Edinburgh,
Mayfield Road, Edinburgh EH9 3JZ, United Kingdom e-mail: K.Fountoulakis@sms.ed.ac.uk.} \and
        Jacek~Gondzio\thanks{J. Gondzio is with the School of Mathematics and Maxwell Institute, The University of Edinburgh, Edinburgh,
Mayfield Road, Edinburgh EH9 3JZ, United Kingdom e-mail: J.Gondzio@ed.ac.uk. J. Gondzio is supported by EPSRC Grant EP/I017127/1}
}

\markboth{Journal of \LaTeX\ Class Files,~Vol.~11, No.~4, December~2012}%
{Shell \MakeLowercase{\textit{et al.}}: Bare Demo of IEEEtran.cls for Journals}

\maketitle

\begin{abstract}
In this paper we are interested in the solution of Compressed Sensing (CS) problems where the signals to be recovered are sparse in coherent and redundant dictionaries. 
CS problems of this type are convex with non-smooth and non-separable regularization term, therefore a specialized solver is required. 
We propose a primal-dual Newton Conjugate Gradients (pdNCG) method.
We prove global convergence and fast local rate of convergence for pdNCG. Moreover, well-known properties of CS problems are exploited 
for the development of provably effective preconditioning techniques that speed-up the approximate solution of linear systems which arise. 
Numerical results are presented on CS problems which demonstrate the performance of pdNCG compared to a state-of-the-art existing solver.

\end{abstract}


\section{Introduction}

CS is concerned with recovering signal $\tilde x\in\mathbb{R}^{n}$ by observing a linear combination of the signal
$$
\tilde b = A \tilde x,
$$
where $A\in \mathbb{R}^{m\times n}$ is an under-determined linear operator with $m < n$ 
and $\tilde b \in \mathbb{R}^{m}$ are the observed measurements. 
Although this system has infinitely many solutions, reconstruction of $\tilde x$ is possible due to its assumed properties.
In particular, $\tilde x$ is assumed to have a sparse image through a redundant and coherent dictionary 
$W\in E^{n\times l}$, where $E=$ $\mathbb{R}$ or $\mathbb{C}$ and $n \le l$. More precisely, $W^* \tilde x$, is sparse, i.e. it has only few non-zero components,
where the star superscript denotes the conjugate transpose.
If $W^* \tilde{x}$ is sparse, then 
the optimal solution of the linear $\ell_1$-analysis problem
$$
\mbox{minimize} \ \|W^* x\|_1, \quad \mbox{subject to:} \quad Ax=\tilde b
$$
is proved to be equal to $\tilde x$, where $\|\cdot\|_1$ is the $\ell_1$-norm. 

Frequently measurements $\tilde b$ might be contaminated with noise, i.e. one measures $b = \tilde b + e$ instead, 
where $e$ is a vector of noise, usually modelled as Gaussian with zero-mean and bounded Euclidean norm.
In addition, in realistic applications, $W^*\tilde x$ might not be exactly sparse, but its mass might be concentrated only on few of its components, while the rest
are rapidly decaying. 
In this case, the optimal solution of the following $\ell_1$-analysis problem
\begin{equation}\label{prob1}
\mbox{minimize} \  f_{c}(x) :=  c\|W^* x\|_1 + \frac{1}{2}\|Ax-b\|^2_2,
\end{equation}
is proved to be a good approximation to $\tilde x$. In \eqref{prob1}, $c$ is an a-priori chosen positive scalar and $\|\cdot\|_2$ is the Euclidean norm.
Discussion on conditions that guarantee the reconstruction of $\tilde x$ are restated in Subsection \ref{subsec:wrip}.

\subsection{Brief Description of CS Applications}\label{subsec:apps} 
An example of $W$ being redundant and coherent with orthonormal rows is the curvelet frame where an image is assumed to have an approximately sparse representation \cite{apps3}. 
Moreover, for radar and sonar systems it is frequent that Gabor frames are used
in order to reconstruct pulse trains from CS measurements \cite{apps4}.
For more applications a small survey is given in \cite{l1analysis}.
Isotropic Total-Variation (iTV) is another application of CS, which exploits the fact that digital images frequently have slowly varying pixels, except along edges. This property implies that digital images 
with respect to the discrete nabla operator, i.e. local differences of pixels, are approximately sparse.
For iTV applications, matrix $W\in \mathbb{C}^{n \times n}$ is square, complex and rank-deficient with $rank (W)=n-1$.
An alternative to iTV is $\ell_1$-analysis, where matrix $W$ is a Haar wavelet transform. However, it has been stated in \cite{tvrobust}, that compared to the $\ell_1$-analysis problem, a more pleasant to the eye reconstruction is obtained by solving the iTV problem.



\subsection{Conditions and Properties of Compressed Sensing Matrices}\label{subsec:wrip}

There has been an extensive amount of literature studying conditions and properties of matrices $A$ and $W$ which guarantee recoverability of a good approximation of $\tilde x$ by solving problem \eqref{prob1}.
For a thorough analysis we refer the reader to \cite{l1analysis,tvrobust}.
The previously cited papers use
a version of the well-known \textit{Restricted Isometry Property} (RIP) \cite{l1analysis}, which is repeated below.
\begin{definition}\label{def:1}
The restricted isometry constant of a matrix $A\in \mathbb{R}^{m\times n}$ adapted to $W\in {E}^{n\times l}$ is defined as the smallest $\delta_q$ such that
$$
(1-\delta_q)\|Wz\|_2^2 \le \|AWz\|_2^2 \le (1+\delta_q)\|Wz\|_2^2
$$
for all at most $q$-sparse $z\in {E}^{l}$, where $E=\mathbb{R} \mbox{ or } \mathbb{C}$.
\end{definition}

For the rest of the paper we will refer to Definition \ref{def:1} as W-RIP.
It is proved in Theorem $1.4$ in \cite{l1analysis} that if $W\in E^{n\times l}$ has orthonormal rows with $n \le l$ and if $A$, $W$ satisfy the W-RIP with $\delta_{2q} \le 8.0e$-$2$, then the solution $x_c$ obtained 
by solving problem \eqref{prob1}
satisfies
\begin{equation}\label{eq:160}
\|x_c - \tilde x\|_2 = C_0 \|e\|_2  + C_1 \frac{\|W^*x_c - (W^*\tilde x)_q\|_1}{\sqrt{q}},
\end{equation}
where $(W^*\tilde x)_q$ is the best $q$-sparse approximation of $W^*\tilde x$, $C_0$ and $C_1$ are small constants and only depend on $\delta_{2q}$. 
It is clear that $W^* \tilde x$ must have $l-q$ rapidly decaying components, in order for 
$\|x_c - \tilde x\|_2$ to be small and the reconstruction to be successful. 
iTV is a special case of $\ell_1$-analysis where matrix $W$ does not have orthonormal rows, hence, result \eqref{eq:160} does not hold. For iTV
there are no conditions on $\delta_{2q}$ such that a good reconstruction is assured. However, there exist results which directly impose restrictions on the number of measurements $m$, see Theorems $2$, $5$ and $6$ in \cite{tvrobust}.
Briefly, in these theorems it is mentioned that if $m \ge q \log(n)$ linear measurements are acquired for which matrices $A$ and $W$ satisfy the W-RIP for some $\delta_q < 1$, then, similar reconstruction guarantees 
as in \eqref{eq:160} are obtained for iTV.
Based on the previously mentioned results regarding reconstruction guarantees it is natural to assume that for iTV a similar condition applies, i.e. $\delta_{2q} < 1/2$. Hence, we make the following assumption.
\begin{assumption}\label{assum:1}
The number of nonzero components of $W^*x_c$, denoted by $q$, and the dimensions $l$, $m$, $n$ are such that matrices $A$ and $W$ satisfy W-RIP for some $\delta_{2q} < 1/2$. 
\end{assumption}
This assumption will be used in the spectral analysis of our preconditioner in Section \ref{sec:cs}.
Another property of matrix $A$ is the near orthogonality of its rows.
Indeed many applications in CS use matrices $A$ that satisfy
\begin{equation}\label{bd8}
\| AA^\intercal - I_m \|_2 \le \delta,
\end{equation}
with a small constant $\delta\ge 0$.
Finally, through the paper we will make use of the following assumption  
\begin{equation}\label{bd9}
\mbox{Ker}(W^*_{\mathcal{N}})\cap \mbox{Ker}(A) = \{0\},
\end{equation}
where $\mathcal{N} := \{j\in\{1,2,\ldots, l\}  \ | \ j\notin \mbox{supp}(W^* x_c)\}$, $\mbox{supp}(x) := \{ i\in\{1,2,\cdots,n\} \ | \ x\in\mathbb{R}^n, x_i \neq 0 \}$. 
Condition \eqref{bd9} implies that
\begin{equation}\label{bd54}
\mbox{Ker}(W^*)\cap \mbox{Ker}(A) = \{0\},
\end{equation}
because $\mbox{Ker}(W^*) \subset \mbox{Ker}(W^*_{\mathcal{N}})$. 
These are commonly used assumptions in the literature, see for example \cite{fadili}. Both are crucial for the analysis of 
the proposed method. 

\subsection{Contribution}
In \cite{ctpdnewton}, Chan, Golub and Mulet, proposed a primal-dual Newton Conjugate Gradients method for image denoising and deblurring problems.
In this paper we modify their method and adapt it for CS problems with coherent and redundant dictionaries. There are three major contributions.

First, we present a complete convergence theory for the proposed pdNCG. In particular, we prove global convergence and local superlinear rate of convergence
for the non strongly convex problems which arise. To the best of our knowledge such an analysis is not available in the current literature for pdNCG. 

Second, we propose an inexpensive preconditioner for fast solution of systems in pdNCG when applied on CS problems with coherent and redundant dictionaries.
We analyze the limiting behaviour of our preconditioner and prove that the eigenvalues of the preconditioned matrices are clustered around one.
This is an essential property that guarantees that only few iterations of CG will be needed to solve approximately the linear systems. Moreover, 
we provide computational evidence that the preconditioner works well not only close to the solution (as predicted by its spectral analysis) but also 
in earlier iterations of pdNCG.

Third, we demonstrate that despite being a second-order method, pdNCG can be more efficient than a specialized first-order method for CS problems of our interest,
even on large-scale instances.
This performance is observed in several numerical experiments presented in this paper. 
We believe that the reason for this is that pdNCG, as a second-order method, captures the curvature of the problems, which results in sufficient decrease in the number of iterations compared 
to a first-order method. This advantage comes with the computational cost of having to solve a linear system at every iteration. 
However, inexact solution of the linear systems using CG combined with the proposed
efficient preconditioner crucially reduces the computational costs per iteration.

\subsection{Format of the Paper and Notation}
The paper is organized as follows. 
In Section \ref{sec:smooth}, problem \eqref{prob1} is replaced by a smooth approximation; the $\ell_1$-norm is approximated by the pseudo-Huber function.
Moreover, some properties of the perturbed objective function and the perturbed optimal solution are shown.
In Section \ref{sec:pdNCG}, pdNCG is presented.
In Section \ref{sec:convanalysis}, global convergence is proved for pdNCG and fast local rate of convergence is established.
In Section \ref{sec:cs}, preconditioning techniques are described for controlling the spectrum of matrices in systems which arise. 
In Section \ref{sec:cont}, a continuation framework for pdNCG is described.
In Section \ref{secNumExp}, numerical experiments are discussed that present the efficiency of pdNCG. Finally, in Section \ref{sec:concl}, conclusions are made. 

Throughout the paper, $\|\cdot\|_1$ is the $\ell_1$-norm, $\|\cdot\|_2$ is the Euclidean norm and $\|\cdot\|_\infty$ the infinity norm. The functions $Re(\cdot)$ and $Im(\cdot)$
take a complex input and return its real and imaginary part, respectively. For simplification of notation, occasionally we will use $Re(\cdot)$ and $Im(\cdot)$ without the parenthesis. 
Furthermore, $diag(\cdot)$ denotes the function which takes as input a vector and outputs a diagonal square matrix with the vector in
the main diagonal. Finally, the super index $c$ denotes the complementarity set, i.e. $\mathcal{J}^c$
is the complementarity set of $\mathcal{J}$.


 


\section{Regularization by Pseudo-Huber}\label{sec:smooth}
In this paper we choose to deal with the non-differentiability of the $\ell_1$-norm by applying smoothing. 
This gives us direct access to second-order information of problem \eqref{prob1}. Moreover, as we shall see
later it allows for the development of a primal-dual method where the steps are taken in both the primal and dual spaces simultaneously.
 These two properties are very important for the robustness of the proposed method. 

In order to perform smoothing, the $\ell_1$-norm is replaced with the pseudo-Huber function \cite{Hartley2004}
\begin{equation}\label{bd6}
 \psi_\mu(W^* x) := \sum_{i=1}^l (({\mu^2+{|W_i^* x|^2}})^{\frac{1}{2}} - \mu),
\end{equation}
where $W_i$ is the $i^{th}$ row of matrix $W\in E^{n\times l}$ and $\mu$ controls the quality of approximation, i.e. for $\mu\to 0$, $\psi_\mu(x)$ tends to the $\ell_1$-norm.
The Pseudo-Huber function is smooth and has derivatives of all degrees. It can be derived by perturbing  
the absolute value function $|x|=\sup\{xz \ | \ -1\le z \le 1 \}$ with the proximity function $d(z)=1-\sqrt{1-z^2}$ in order to get
the smooth function 
$$|x|_\mu = \sup\{xz + \mu\sqrt{1-z^2} -\mu\ | \ -1\le z \le 1 \}=\sqrt{x^2 + \mu^2} -\mu.$$
The original problem \eqref{prob1} is approximated by
\begin{equation}\label{prob2}
\mbox{minimize} \  f_c^\mu(x) :=  c\psi_\mu(W^* x) + \frac{1}{2}\|Ax-b\|^2_2.
\end{equation}

\subsection{Derivatives of Perturbed Function}
The gradient of pseudo-Huber function $\psi_{\mu}(W^* x)$ in \eqref{bd6} is given by
\begin{equation}\label{bd58}
\nabla \psi_{\mu}(W^* x) = (ReW D ReW^\intercal +  ImW D ImW^\intercal)x,
\end{equation}
where $D := diag(D_1, D_2, \cdots, D_l)$ with
\begin{equation}\label{bd61}
D_i := (\mu^2+|y_i|^2)^{-\frac{1}{2}} \quad \forall i=1,2,\cdots,l,
\end{equation}
and $y =[y_1,y_2,\cdots,y_l]^\intercal := W^*x$.
The gradient of function $f_c^\mu(x)$ in \eqref{prob2} is
\begin{equation}\label{bd101}
\nabla f_c^\mu(x) = c\nabla \psi_\mu(W^*x) + A^\intercal(Ax - b).
\end{equation}
The Hessian matrix of $\psi_\mu(x)$ is
\begin{equation}\label{nabla2psi}
\nabla^2\psi_{\mu}(W^*x) :=\frac{1}{4}(W \hat{Y} W^* + \bar W \hat{ {Y}} \bar W^* + W \tilde{Y} \bar W^* + \bar W \tilde{\bar Y} W^*),
\end{equation}
where the bar symbol denotes the complex conjugate, 
$\hat Y :=diag\left[\hat{Y}_1,\hat{Y}_2,..., \hat{Y}_l\right]$, 
$\tilde{Y} :=diag\left[\tilde{Y}_1,\tilde{Y}_2,..., \tilde{Y}_l\right]$ and
\begin{equation}\label{bd62}
\hat{Y}_i:={\mu^2}D_i^{3}+D_i, \quad 
\tilde{Y}_i:=-{y_i^2}D_i^{3},\quad i=1,2,...,l,
\end{equation}
Moreover, the Hessian matrix of $f_c^\mu(x)$ is
\begin{equation}\label{eq131}
\nabla^2 f_c^\mu(x) = c\nabla^2 \psi_\mu(W^*x) + A^\intercal A.
\end{equation}

\subsection{Continuous path}
In the following lemma we show that $x_{c,\mu}$ for $c$ constant is a continuous and differentiable function of $\mu$.
\begin{lemma}\label{lem:7}
Let $c$ be constant and consider $x_{c,\mu}$ as a functional of $\mu>0$. If
$
\mbox{Ker}(W^*)\cap\mbox{Ker}(A)=\{0\}
$,
$x_{c,\mu}$ is continuous and differentiable.
\end{lemma}
\begin{proof}
The optimality conditions of problem \eqref{prob2} are 
$$
c\nabla \psi_\mu(W^*x) +A^\intercal (Ax-b) = 0.
$$ 
According to definition of $x_{c,\mu}$, we have 
\begin{align*}
c\nabla \psi_\mu(W^*x_{c,\mu}) +A^\intercal (Ax_{c,\mu}-b) & = 0 & \Longrightarrow \\
c \frac{d \nabla \psi_\mu(W^*x_{c,\mu}) }{d \mu} + A^\intercal A \frac{d x_{c,\mu}}{d\mu} & =0 & \Longrightarrow \\
c \Big(\nabla^2 \psi_\mu(W^*x_{c,\mu}) \frac{d x_{c,\mu}}{d\mu} + \frac{d \nabla \psi_\mu(W^*x)}{d\mu} \Big|_{x_{c,\mu}}\Big) + A^\intercal A \frac{dx_{c,\mu}}{d\mu} & =0 & \Longleftrightarrow \\ 
\Big(c\nabla^2 \psi_\mu(W^*x_{c,\mu}) + A^\intercal A\Big)\frac{d x_{c,\mu}}{d\mu} + c\frac{d \nabla \psi_\mu(W^*x)}{d\mu} \Big|_{x_{c,\mu}} & =0 &\Longleftrightarrow \\
\nabla^2 f_c^\mu(W^*x_{c,\mu})\frac{d x_{c,\mu}}{d\mu} + c\frac{d \nabla \psi_\mu(W^*x)}{d\mu} \Big|_{x_{c,\mu}} & =0, & 
\end{align*}
where ${d \nabla \psi_\mu(W^*x)}/{d\mu} |_{x_{c,\mu}}$ is the first-order derivative of $\nabla \psi_\mu(W^*x)$ as a functional of $\mu$, measured at $x_{c,\mu}$. Notice that due to condition 
$
\mbox{Ker}(W^*)\cap\mbox{Ker}(A)=\{0\}
$ 
we have that $\nabla^2 f_c^\mu(x)$ is positive definite $\forall x$, hence $x_{c,\mu}$ is unique. Therefore, the previous system has a unique solution, which means that
$x_{c,\mu}$ is uniquely differentiable as a functional of $\mu$ with $c$ being constant. Therefore, $x_{c,\mu}$ is continuous as a functional of $\mu$. 
\end{proof}

\begin{remark}\label{rem:3}
Lemma \ref{lem:7} and continuity imply that there exists sufficiently small smoothing parameter $\mu$ such that $\|x_{c,\mu}-x_c\|_2< \omega$ for any arbitrarily small $\omega>0$.
\end{remark}

\subsection{Lipschitz Continuity of Hessian of Perturbed Function}\label{sec:prela}
In this subsection Lipschitz continuity of the Hessian of pseudo-Huber function
and the Hessian of the perturbed function $f_c^\mu(x)$ are proved. These results will be used in proving fast local rate of convergence of pdNCG.

\begin{lemma} \label{lem:2}
The Hessian matrix $\nabla^2\psi_{\mu}(W^*x)$ is Lipschitz continuous
\[
\|\nabla^2\psi_{\mu}(W^*y)-\nabla^2\psi_{\mu}(W^*x)\|_2\leq L_\psi\|y-x\|_2,
\]
where $L_\psi>0$.
\end{lemma}
\begin{proof}
Let $u=W^*x$, $v=W^*y$ and $z(s) = u + s(v-u)$, then by using Cauchy-Schwartz we have that
\begin{align*}
\|\nabla^2\psi_{\mu}(W^*y)-\nabla^2\psi_{\mu}(W^*x)\|_2 & =\Big\|\int^{1}_{0}\frac{\nabla^2\psi_{\mu}(z(s))}{ds}ds\Big\|_2 \le \int^{1}_{0}\Big\|\frac{\nabla^2\psi_{\mu}(z(s))}{ds}\Big\|_2ds.
\end{align*}
Furthermore by using the above expession and \ref{nabla2psi} we get
\begin{align}\label{lm6.1}
\left\|\nabla^2\psi_{\mu}(W^*y)-\nabla^2\psi_{\mu}(W^*x)\right\|_2 & \leq 
											   \frac{1}{4}\left\|\left[\begin{array}{cc}W& \bar W \end{array}\right]\right\|_2\left\|\left[\begin{array}{c}W^*\\ \bar W^* \end{array}\right]\right\|_2\int^{1}_{0}\left\|\frac{d}{ds}\hat Z\right\|_2ds \nonumber \\
									           &  + \frac{1}{4}\left\|\left[\begin{array}{cc}W&\bar W \end{array}\right]\right\|_2 \left\|\left[\begin{array}{c} \bar W^*\\W^* \end{array}\right]\right\|_2 \int^{1}_{0}\left\|\frac{d}{ds}\tilde{Z}\right\|_2ds,
\end{align}
where $\hat Z = diag[\hat{Y}_1,...,\hat{Y}_l,\hat{Y}_1,...,\hat{Y}_l]$ and $\tilde{Z}=diag[\tilde{Y}_1,..., \tilde{Y}_l, \tilde{\bar Y}_1,..., \tilde{\bar Y}_l]$, 
which according to \eqref{bd62} and \eqref{bd61} are equal to  
$$\hat{Z}_i=\frac{\mu^2}{(\mu^2+\left|[z(s)]_i\right|^2)^{\frac{3}{2}}}+\frac{1}{(\mu^2+\left|[z(s)]_i\right|^2)^{\frac{1}{2}}}, \quad \forall i=1,2,...,2l$$
and
$$\tilde{Z}_i=-\frac{[z(s)]_i^2}{(\mu^2+\left|[z(s)]_i\right|^2)^\frac{3}{2}}, \quad
\tilde{\bar Z}_i=-\frac{[\bar z(s)]_i^2}{(\mu^2+\left|[z(s)]_i\right|^2)^\frac{3}{2}}, \quad \forall i=1,2,...,l.$$
Furthermore,
\[
\left\|\frac{d}{ds}\hat Z\right\|_2=\left\|vec\left(\frac{d}{ds}\hat Z\right)\right\|_\infty=\displaystyle\max\left|\left[\frac{d}{ds}\hat Z\right]_{ii}\right|=\displaystyle\max\left|\frac{d}{ds}\hat Z_i\right|,
\]
where 
$$
\frac{d}{ds}\hat Z_i=-[Re(u_i-v_i)[\bar{z}(s)]_i]\left[\frac{\mu}{(\mu^2+\left|[z(s)]_i\right|^2)^{\frac{3}{2}}}+\frac{3\mu^3}{(\mu^2+\left|[z(s)]_i\right|^2)^{\frac{5}{2}}}\right]
$$
and $vec\left(\frac{d}{ds}\hat Z\right)$ is the vectorization of a matrix $\frac{d}{ds}\hat Z$.
Thus
\begin{equation}\label{lm6.2}
\left\|\frac{d}{ds}\hat Z\right\|\leq \frac{1}{{\mu}^2}\left\|v-u\right\|\displaystyle\max_{i} M_i<\frac{1}{\mu}\left\|v-u\right\|,
\end{equation}
where 
$$
M_i=\left|[z(s)]_i\right|\left[\frac{\mu^5}{|\mu^2+\left|[z(s)]_i\right|^2|^{\frac{5}{2}}}+3\frac{\mu^7}{|\mu^2+\left|[z(s)]_i\right|^2|^{\frac{7}{2}}}\right]
$$ 
and $\displaystyle\max_{i} M_i\leq\frac{4\mu}{9\sqrt{3}}+\frac{192\mu}{125\sqrt{5}}<\mu$. Moreover,
\[
\left\|\frac{d}{ds}\tilde{Z}\right\|_2=\left\|vec\left(\frac{d}{ds}\tilde{Z}\right)\right\|_\infty=
\displaystyle\max\left|\left[\frac{d}{ds}\tilde{Z}\right]_{ii}\right|=\displaystyle\max\left\{\left|\frac{d}{ds}\tilde{Z}_i\right|,\left|\frac{d}{ds}\tilde{\bar Z}_i\right|\right\},
\]
where 
$$
\frac{d}{ds}\tilde {Z}_i=-\mu\frac{2[z(s)]_i(v_i-u_i)[\mu^2+\left|[z(s)]_i\right|^2]-3[z(s)]_i^2Re([z(s)]_i(v_i-u_i))}{(\mu^2+\left|[u(s)]_i\right|^2)^\frac{5}{2}}
$$
and thus
\begin{equation}\label{eq:170}
\left\|\frac{d}{ds}\tilde{Z}_i\right\|\leq\frac{2}{\mu^2}\left\|v-u\right\|
\displaystyle\max_iN_i<\frac{1}{\mu}\left\|v-u\right\|,
\end{equation}
where 
$$
N_i=\mu^3\frac{|[z(s)]_i|(\mu^2+|[z(s)]_i|^2)+\frac{3}{2}|[z(s)]_i|^3}{(\mu^2+|[z(s)]_i|^2)^\frac{5}{2}}
$$ 
and $\displaystyle\max_iN_i\leq \frac{2\mu}{3\sqrt{3}}+\frac{9\sqrt{3}\mu}{125\sqrt{5}}<\frac{1}{2}\mu$.
By following similar reasoning we get
\begin{equation}\label{eq:171}
\left\|\frac{d}{ds}\tilde{\bar Z}_i\right\|<\frac{1}{\mu}\left\|v-u\right\|.
\end{equation}
Hence by using \eqref{eq:170} and \eqref{eq:171} have
\begin{equation}\label{lm6.3}
\left\|\frac{d}{ds}\tilde{Z}\right\|<\frac{1}{\mu}\left\|v-u\right\|,
\end{equation}
From \eqref{lm6.1}, \eqref{lm6.2} and \eqref{lm6.3}
\[
\left\|\nabla^2\psi_{\mu}(W^*y)-\nabla^2\psi_{\mu}(W^*x)\right\|\leq\frac{1}{\mu^2}L_\psi\left\|y-x\right\|,
\]
where 
\[
L_\psi=\frac{1}{4}\left\|W^*\right\|\left(
\left\|\left[\begin{array}{cc}W&\bar W \end{array}\right]\right\|
\left\|\left[\begin{array}{c}W^*\\\bar W^* \end{array}\right]\right\|
+\left\|\left[\begin{array}{cc}W&\bar W \end{array}\right]\right\|
\left\|\left[\begin{array}{c} \bar W^*\\W^* \end{array}\right]\right\|\right).
\]
\end{proof}

In the following lemma it is shown that the Hessian matrix of $f_c^\mu(x)$ in \eqref{eq131} is Lipschitz continuous.
\begin{lemma} \label{lem:4}The Hessian matrix $\nabla^2 f_c^\mu(x)$ is Lipschitz continuous
\[
\|\nabla^2f_c^{\mu}(y)-\nabla^2f_c^{\mu}(x)\|
\leq L_{f}\|y-x\|,
\]
where $L_f:=cL_\psi$, $L\psi$ is defined in Lemma \ref{lem:2} and $c>0$.
\end{lemma}
\begin{proof}
We have that
$
\|\nabla^2f_\tau^{\mu}(y)-\nabla^2f_\tau^{\mu}(x)\| = c\|\nabla^2\psi_{\mu}(W^*y)-\nabla^2\psi_{\mu}(W^*x)\|.
$
By using Lemma \ref{lem:2} and setting $L_f=c L_{\psi}$we obtain the result.
\end{proof}

\section{Primal-Dual Newton Conjugate Gradients Method}\label{sec:pdNCG}
The method which will be discussed in this section is similar to the second-order primal-dual method proposed in \cite{ctpdnewton}. 
In this paper we demonstrate that it can be applied to much more general CS problems. Moreover, we prove global convergence and fast local rate of convergence of pdNCG for CS problems.
Finally, we propose CS preconditioners which provably improve the performance of CG. 

\subsection{Alternative Optimality Conditions}
In \cite{ctnewtonold} the authors solve iTV problems for square and full-rank matrices $A$ which are inexpensively diagonalizable, i.e. image deblurring or denoising.  
More precisely, in the previous cited paper the authors tackled iTV problems using a Newton-CG method for finding roots of $\nabla f_c^\mu(x) = 0$. 
They observed that close to the points of non-smoothness of the $\ell_1$-norm, the smooth pseudo-Huber function \eqref{bd6}
exhibited an ill-conditioning behaviour. This results in two major drawbacks of the application of Newton-CG. First, the linear algebra is difficult to tackle. Second,
the region of convergence of Newton-CG is substantially shrunk. To deal with these problems they have proposed to 
incorporate Newton-CG inside a continuation procedure on the parameters $c$ and $\mu$. Although they showed that continuation did improve the global
convergence properties of Newton-CG it was later discussed in \cite{ctpdnewton} (for the same iTV problems) that continuation is difficult to control (especially for small $\mu$) and Newton-CG is not always convergent in reasonable CPU time. In the latter paper the authors 
have provided numerical evidence that the behaviour of a Newton-CG method is made significantly more robust, even for small values of $\mu$, by simply solving an equivalent optimality conditions instead. 
For problem \eqref{prob2} which is of our interest, by setting $g_{re} = D ReW^\intercal x$, $g_{im} = D ImW^\intercal x$ and using \eqref{bd58} to the optimality conditions of the perturbed problem \eqref{prob2}, the 
equivalent optimality conditions are
\begin{equation}\label{bd53}
\begin{aligned}
c(ReWg_{re}+ImWg_{im}) +A^\intercal (Ax-b) = 0, \\ 
D^{-1}g_{re} = ReW^\intercal x, \quad D^{-1}g_{im} = ImW^\intercal x.
\end{aligned}
\end{equation}

\subsection{The Method}

At every iteration of pdNCG the primal-dual directions are calculated by approximate solving the following linearization of the equality constraints in \eqref{bd53} 
\begin{equation}\label{bd56}
\begin{aligned}
B\Delta x & =  -\nabla f_c^\mu(x) \\ 
\Delta g_{re}   & = D(I-B_1)ReW^\intercal \Delta x + DB_2ImW^\intercal \Delta x- g_{re} + DReW^\intercal x\\
\Delta g_{im}   & = D(I-B_4)ImW^\intercal \Delta x + DB_3ReW^\intercal \Delta x- g_{im} + DImW^\intercal x
\end{aligned}
\end{equation}
where 
\begin{equation}\label{bd55}
B := c\tilde{B} + A^\intercal A,
\end{equation} 
\begin{align*}
\tilde{B}  & := ReWD(I-B_1)ReW^\intercal + ImWD(I-B_4)ImW^\intercal + ReW DB_2 ImW^\intercal  \\ 
              &  \quad \ + ImW B_3D ReW^\intercal,
\end{align*}
and $B_i, i=1,2,3,4$ are diagonal matrices with components
\begin{eqnarray*}
&[B_1]_{ii} := D_{i}[g_{re}]_iReW_i^\intercal x,&  \quad [B_2]_{ii} := D_{i}[g_{re}]_iImW_i^\intercal x, \\
&[B_3]_{ii} := D_{i}[g_{im}]_iReW_i^\intercal x,& \quad  [B_4]_{ii} := D_{i}[g_{im}]_iImW_i^\intercal x.
\end{eqnarray*}
\begin{remark}\label{rem:1}
Matrix $B$ in \eqref{bd55} is positive definite if $\|g_{re} + \sqrt{-1}g_{im}\|_\infty \le 1$ and \eqref{bd54} are satisfied.
The former condition will be maintained through all iterations of pdNCG. The latter condition holds due to the assumed condition \eqref{bd9}.
\end{remark}

It is straightforward to show the claims in Remark \ref{rem:1} for the case of $W$ being a real matrix. For the case of complex $W$ we refer
the reader to a similar claim which is made in \cite{ctpdnewton}, page $1970$.
Although matrix $B$ is positive definite under the conditions stated in Remark \ref{rem:1}, it is not symmetric, except in the case that $W$ is real where all imaginary parts are dropped.
Therefore in the case 
of complex matrix $W$, preconditioned CG (PCG) cannot be employed to solve approximately \eqref{bd56}.
To avoid the problem of non-symmetric matrix $B$ the authors in \cite{ctpdnewton} have suggested to ignore the non-symmetric part in matrix $B$ and employ
CG to solve \eqref{bd56}. This idea is based on the fact that as the method converges, then the symmetric part of $B$ tends to the symmetric second-order derivative of $f_c^\mu(x)$
(we prove this in Lemma \ref{lem:30}). 
In this paper, we will also follow this approach. The system \eqref{bd56} is replaced with 
\begin{equation}\label{eq108}
\begin{aligned}
\hat{B}\Delta x  &  =  -\nabla f_c^\mu(x) \\
\Delta g_{re}   & = D(I-B_1)ReW^\intercal \Delta x + DB_2ImW^\intercal \Delta x- g_{re} + DReW^\intercal x\\
\Delta g_{im}   & = D(I-B_4)ImW^\intercal \Delta x + DB_3ReW^\intercal \Delta x- g_{im} + DImW^\intercal x
\end{aligned}
\end{equation}
where
\begin{equation}\label{eq109}
\hat{B} := c\,\mbox{sym}(\tilde{B}) + A^\intercal A
\end{equation} 
and $\mbox{sym}(\tilde{B}):= 1/2(\tilde{B} + \tilde{B}^\intercal)$ is the symmetric part of $\tilde{B}$. Moreover, PCG is terminated when 
\begin{equation}\label{bd59}
\|\hat{B}\Delta x + \nabla f_c^\mu(x)\|_2 \le \eta \|\nabla f_c^\mu(x)\|_2,
\end{equation}
is satisfied for $\eta \in [0,1)$. 
Then the iterate $g:=g_{re} + \Delta g_{re} + \sqrt{-1}(g_{im} + \Delta g_{im})$ is orthogonally projected on the box $\{x:\left\|x\right\|_\infty\leq 1\}$.
The projection operator for complex arguments is applied component-wise and it is defined as
$
v :=  P_{\|\cdot\|_\infty\le1}(u) = \mbox{min}({1}/{|u|},1)\odot u
$, where $\odot$ denotes the component-wise multiplication.
In the last step, line-search is employed for the primal $\Delta x$ direction in order to guarantee that the objective value $f_c^\mu(x)$ is monotonically decreasing.
The pseudo-code of pdNCG is presented in Figure \ref{fig:2}.

\begin{figure}
\begin{algorithmic}[1]
\vspace{0.1cm}
\STATE \textbf{Input:} $\tau_1\in(0,1)$, $\tau_2\in(0,1/2)$, $x^0$, $g_{re}^0$ and $g_{im}^0$, where $\|g_{re}^0 + \sqrt{-1}g_{im}^0\|_\infty \le 1$.
\STATE \textbf{Loop:} For $k=1,2,...$, until termination criteria are met. \vspace{0.1cm}
\STATE \hspace{0.5cm} Calculate $\Delta x^k$, $\Delta g_{re}^k$ and $\Delta g_{im}^k$ by solving approximately the system \eqref{eq108}, \\ \vspace{0.1cm}
              \hspace{0.5cm} until \eqref{bd59} is satisfied for some $\eta\in[0,1)$. \\ \vspace{0.1cm}
\STATE \hspace{0.5cm} Set 
$
\tilde{g}_{re}^{k+1}  :=  g_{re}^k + \Delta g_{re}^k 
$
,
$
\tilde{g}_{im}^{k+1}   :=  g_{im}^k + \Delta g_{im}^k
$
and calculate
\begin{equation*}
\bar{g}^{k+1}  :=  P_{\|\cdot\|_\infty\le1}(\tilde{g}_{re}^{k+1}  + \sqrt{-1}\tilde{g}_{im}^{k+1} ),
\end{equation*}
\hspace{0.5cm} where $P_{\|\cdot\|_\infty\le1}(\cdot)$ is the orthogonal projection on the $\ell_\infty$ ball. \\ \vspace{0.1cm} 
\hspace{0.5cm} Then set
${g}_{re}^{k+1} := Re \bar{g}^{k+1} $ and $\quad {g}_{im}^{k+1}:=Im\bar{g}^{k+1} $.
\vspace{0.1cm}
\STATE \hspace{0.5cm} Find the least integer $j\ge0$ such that
\begin{equation*}
f_c^\mu(x^k + \tau_1^j \Delta x^k) \le f_c^\mu(x^k) - \tau_2\tau_1^j \|\Delta x^k\|_{\hat{B}^k}
\end{equation*}
 \hspace{0.5cm} and set $\alpha := \tau_1^j$, where $\|\Delta x^k\|_{\hat{B}^k} = (\Delta x^k)^\intercal \hat{B}^k \Delta x^k$ and $\hat{B}^k:=\hat{B}(x^k,g_{re}^k,g_{im}^k)$.
\STATE \hspace{0.5cm}  Set $x^{k+1} := x^k + \alpha \Delta x^k$. 
\end{algorithmic}
\caption{Algorithm primal-dual Newton Conjugate Gradients}
\label{fig:2}
\end{figure}

\section{Convergence Analysis}\label{sec:convanalysis}
In this section we prove global convergence of the proposed pdNCG and we establish fast local rate of convergence. 
Throughout the paper we will denote the optimal solutions of problems \eqref{prob1} and \eqref{prob2} as $x_{c} := \argmin_x f_c(x)$ and $x_{c,\mu} := \argmin_x f_c^\mu(x)$, respectively.
Furthermore, we define $B^k:=B(x^k,g_{re}^k,g_{im}^k)$ and $\hat{B}^k:= \hat{B}(x^k,g_{re}^k,g_{im}^k)$, where $B$ and $\hat{B}$ are defined in \eqref{bd55} and \eqref{eq109}, respectively.
\subsection{Global Convergence}
The following technical result is useful for the subsequent lemma and theorem. 
\begin{lemma}\label{lem:21}
Let condition \ref{bd9} hold. Then for all iterates $k$ of pdNCG matrix $\hat{B}^k$ is positive definite.
\end{lemma}
\begin{proof}
From Remark \ref{rem:1} we have that matrix ${B}$ in \eqref{bd55} is positive definite if $\|g_{re} + \sqrt{-1}g_{im}\|_\infty \le 1$ and \eqref{bd54} are satisfied.
This implies that the symmetric part of $B$, which is the matrix $\hat{B}$ in \eqref{eq109}, is also positive definite if the same conditions hold.
Condition \eqref{bd54} holds because of the assumed condition \eqref{bd9}.
According to step $3$ of pdNCG, condition  $\|g_{re}^k + \sqrt{-1}g_{im}^k\|_\infty \le 1$ is always satisfied $\forall k$. Hence, matrix $\hat{B}^k$ is positive definite
for all iterations of pdNCG.
\end{proof}

To prove convergence of the primal variables $x$ for pdNCG method we need to show first that at a point $x$ for which $\nabla f_c^\mu(x)\neq 0$ there exists a step-size $\alpha > 0$
such that the line-search termination condition in step $4$ of pdNCG is satisfied. This is shown in the next lemma.
\begin{lemma}\label{lem:9}
Let condition \eqref{bd9} hold. Moreover, let $x^k$ be the current iterate of pdNCG such that $\nabla f_c^\mu(x^k)\neq 0$ and $\Delta x^k$ be the direction calculated in step $2$ of pdNCG.
If PCG is initialized with the zero solution with termination criterion \eqref{bd59},
then the backtracking line-search algorithm will calculate a step-size ${\alpha}$ such that 
$
{\alpha} \ge {\tau_2}/{\kappa(\hat{B}^k)}
$
and the following holds
$$
f_c^\mu(x^k) - f(x({\alpha})) > \frac{\tau_1 \tau_2}{\kappa(\hat{B}^k)} \|\Delta x^k\|_{\hat{B}^k}, 
$$
where $x(\alpha) = x^k + \alpha \Delta x^k$, $\tau_1$ and $\tau_2$ are used in step $4$ of algorithm pdNCG.
\end{lemma}
\begin{proof}
The proof is very similar to the analysis of Lemma $9$ in \cite{l1regSCfg}, however, it is 
worth going through the most important steps of it again since some details vary. 
From Lemma \ref{lem:21} we have that matrix $\hat{B}^k$ is positive definite $\forall k$.  
According to Lemma $7$ in \cite{l1regSCfg}, if $\nabla f_c^\mu(x^k)\neq 0$ and PCG is initialized with the zero solution, then PCG
at the $i^{th}$ iteration returns the direction $\Delta x^k$ which satisfies
$$
(\Delta x^k)^\intercal \hat{B}^k \Delta x^k = - (\Delta x^k)^\intercal \nabla f_c^\mu(x^k).
$$
Therefore, it holds that 
$$
f_c^\mu(x(\alpha)) \le  f_c^\mu(x^k) - \alpha \|\Delta x^k\|_{\hat{B}^k}^2 + \kappa(\hat{B}^k)\frac{\alpha^2}{2}\|\Delta x^k\|_{\hat{B}^k}^2.
$$
The minimizer of the right hand side of the previous inequality is $\tilde{\alpha}=1/\kappa(\hat{B})$ and it satisfies 
$$
f_c^\mu(x(\tilde{\alpha})) \le f_c^\mu(x^k)  - \frac{\kappa(\hat{B}^k)}{2}\|\Delta x^k\|_{\hat{B}^k}^2,
$$
where $x(\tilde{\alpha})= x^k  + \tilde{\alpha}\Delta x^k$. The step-size $\tilde{\alpha}$ satisfies the termination condition of the line-search method in step $4$ of pdNCG in Figure \ref{fig:2}.
Therefore, in the worst case, the line-search method will return a step-size which cannot be smaller than $\tau_2/\kappa(\hat{B}^k)$. For this minimum
step-size we obtain the minimum decrease in the preamble of this lemma. 
\end{proof}

Based on the previous lemma, in the next theorem we prove convergence of pdNCG for the primal variables.
\begin{theorem}\label{thm:5}
Let condition \eqref{bd9} hold. Moreover, let $\{x^k\}$ be a sequence generated by pdNCG. PCG
is used as described in the preamble of Lemma \ref{lem:9}. Then, the sequence  $\{x^k\}$ converges to $x_{c,\mu}$.
\end{theorem}
\begin{proof}
From Lemma \ref{lem:9} we have that if $x^k$ is not the optimal solution of problem \eqref{prob2}, i.e. $\nabla f_c^\mu(x^k)\neq 0$, then
the objective function $f_c^\mu(x)$ is monotonically decreased when a step is made from $x^k$ to $x^{k+1}$. This implies that if $k\to \infty$ then $f_c^\mu(x^k) - f_c^\mu(x^{k+1}) \to 0$. 
Since $f_c^\mu(x^0) < \infty$ and $f_c^\mu(x)$ is monotonically decreased, where $x^0$ is a finite first guess given as an input to pdNCG, 
then the sequence $\{x^k\}$ belongs in a closed, bounded and therefore, compact sublevel set. Hence, the sequence $\{x^k\}$  must have a subsequence which converges to a point $x^*$
and this implies that $\{x^k\}$ also converges to $x^*$.
From Lemma \ref{lem:21} we have that matrix $\hat{B}^k$ is positive definite $\forall k$.  Since $\hat{B}^k$ is positive definite, 
from Lemma \ref{lem:9} we have that $\Delta x^k \to 0$. 
Moreover, PCG returns the zero direction $\Delta x^k$
if and only if $\nabla f_c^\mu(x^k)=0$. Therefore, for $k\to \infty$ we have that $\nabla f_c^\mu(x^k)\to\nabla f_c^\mu(x^*)=0$, hence, $x^k \to x_{c,\mu}$, which is the optimal solution
of problem \eqref{prob2}. 
\end{proof}

In the following theorem, convergence of the dual variables is established. This theorem in combination with Theorem \ref{thm:5} imply that 
the primal-dual iterates produced by pdNCG converge to a solution of the system \eqref{bd53}.
\begin{theorem}\label{thm:6}
Let the assumptions of Theorem \ref{thm:5} hold. Then we have that the sequences of dual variables produced by pdNCG satisfy 
$\{g_{re}^k\} \to DReW^\intercal x_{c,\mu}$, $\{g_{im}^k\} \to DImW^\intercal x_{c,\mu}$, where $D$ in \eqref{bd61} is measured at $x_{c,\mu}$. 
Furthermore, the previous imply that the primal-dual iterates of
pdNCG converge to the solution of system \eqref{bd53}.
\end{theorem}
\begin{proof}
From Theorem \ref{thm:5} we have that $x^k \to x_{c,\mu}$  and $\Delta x^k \to 0$. Hence, from \eqref{eq108} we get that $\Delta g_{re}^k \to -g_{re}^k + DReW^\intercal x_{c,\mu}$
and $\Delta g_{im}^k \to -g_{im}^k + DImW^\intercal x_{c,\mu}$, where $D$ is defined in \eqref{bd61} and in this case is measured at $x_{c,\mu}$. Moreover, we have that 
the iterates at step $3$ of pdNCG satisfy $\tilde{g}_{re}^k \to  DReW^\intercal x_{c,\mu}$, $\tilde{g}_{im}^k  \to  DImW^\intercal x_{c,\mu}$. Consequently, from step $3$ we have 
\begin{align*}
\bar{g}^k  & =  P_{\|\cdot\|_\infty\le1}(\tilde{g}_{re}^k + \sqrt{-1}\tilde{g}_{im}^k) \\ 
                                           & \to P_{\|\cdot\|_\infty\le1}(DReW^\intercal x_{c,\mu} + \sqrt{-1}DImW^\intercal x_{c,\mu}) \\ 
                                           & = DReW^\intercal x_{c,\mu} + \sqrt{-1}DImW^\intercal x_{c,\mu}.
\end{align*}
The previous means that $g_{re}^k\to DReW^\intercal x_{c,\mu}$ and $g_{im}^k\to DImW^\intercal x_{c,\mu}$. 
It is easy to check that at the limit $k\to\infty$, the values of $g_{re}^k$ and $g_{im}^k$ with the optimal variable $x_{c,\mu}$ satisfy system \eqref{bd53}.
\end{proof}

Based on Theorem \ref{thm:6} we prove in the following lemma that $\hat{B}$ in \eqref{eq109} converges to the second-order derivative of $f_c^\mu(x)$.
We will use this lemma in order to prove local superlinear rate of convergence in the next subsection.
Let us first present an alternative definition to \eqref{nabla2psi} of the Hessian matrix of pseudo-Huber function.
\begin{eqnarray}\label{eq121}
\nabla^2\psi_{\mu}(W^*x)  & =& ReWD(I-\tilde{B}_1)ReW^\intercal + ImWD(I-\tilde{B}_4)ImW^\intercal  \\ 
              & &+ ReW D\tilde{B}_2 ImW^\intercal + ImW \tilde{B}_3D ReW^\intercal \nonumber,
\end{eqnarray}
and $\tilde{B}_i, i=1,2,3,4$ are diagonal matrices with components
\begin{align*}
[\tilde{B}_1]_{ii} &:= D_{i}^2(ReW_i^\intercal x)^2, \quad  [\tilde{B}_2]_{ii} := D_{i}^2 (ReW_i^\intercal x) ImW_i^\intercal x, \\
[\tilde{B}_3]_{ii} &:= D_{i}^2(ImW_i^\intercal x)ReW_i^\intercal x, \quad [\tilde{B}_4]_{ii} := D_{i}^2(ImW_i^\intercal x)^2.
\end{align*}
This definition can be obtained by derivation of \eqref{bd58}. In the next lemma notice from \eqref{eq108} and step $3$ of pdNCG, 
that the dual iterates $g_{re}^k$ and $g_{im}^k$ depend on $x^k$. 
\begin{lemma}\label{lem:30}
Let the assumptions of Theorem \ref{thm:5} hold. Let the primal-dual sequences $\{x^k\}$, $\{g_{re}^k\}$ and $\{g_{im}^k\}$ be produced by pdNCG.
If PCG is initialized with the zero solution with termination criterion \eqref{bd59}, then $\hat{B}^k \to \nabla^2 f_c^\mu(x_{c,\mu})$
for $k\to \infty$.
\end{lemma}
\begin{proof}
From Theorem \ref{thm:6} we have that $g_{re}^k \to DReW^\intercal x_{c,\mu}$ and $g_{im}^k \to DImW^\intercal x_{c,\mu}$ for $k\to \infty$,
where $D$ is defined in \eqref{bd61} and is measured at $x_{c,\mu}$.
Using this in the definition of $B$ in \eqref{bd55} we have that ${B}^k$ tends to the symmetric
matrix $\nabla^2 f_c^\mu(x_{c,\mu})$. Therefore,  since $\hat{B}^k$ is the symmetric part of $B^k$, we have that $\hat{B}^k \to \nabla^2 f_c^\mu(x_{c,\mu}) $.
\end{proof}

\subsection{Local Rate of Convergence}
The following lemma shows that the length of the primal directions $\|\Delta x^k\|_2$ calculated in step $2$ of pdNCG
is of order $\|\nabla f_c^\mu(x^k)\|_2$.
\begin{lemma}\label{lem:14}
Let condition \eqref{bd9} hold. Let $\Delta x^k$ be the primal direction calculated in step $2$ of pdNCG, which satisfies \eqref{bd59}. Then
$
\|\Delta x^k\|_2 = \mathcal{O}(\|\nabla f_c^\mu(x^k)\|_2)
$
for $k\to \infty$.
\end{lemma}
\begin{proof}
Let $\tilde{r}_c^\mu(x^k):=\hat{B}^k\Delta x^k + \nabla f_c^\mu(x^k)$. Then
using Cauchy-Schwartz on $\Delta x^k  = (\hat{B}^k)^{-1}(-\nabla f_c^\mu(x^k) + \tilde{r}_c^\mu(x^k)) $ we get
\begin{align*}
\|\Delta x^k\|_2 & \le \|(\hat{B}^k)^{-1}\|_2\|-\nabla f_c^\mu(x^k) + \tilde{r}_c^\mu(x^k)\|_2.
\end{align*}
From \eqref{bd59} and $\eta < 1$ we have that $\|\tilde{r}_c^\mu(x^k)\|_2 \le \|\nabla f_c^\mu(x^k)\|_2$. Hence,
$
 \|\Delta x^k\|_2 \le 2\|(\hat{B}^k)^{-1}\|_2 \|\nabla f_c^\mu(x^k)\|_2.
$
From Lemma \ref{lem:21} we have that $\hat{B}^k$ is positive definite $\forall k$. Therefore
$\|(\hat{B}^k)^{-1}\|_2$ is bounded as $k\to \infty$ and
the result in the preamble of this lemma holds.
\end{proof}

We now have all the tools to establish local superlinear rate of convergence of pdNCG. 
\begin{theorem}\label{thm:30}
Let condition \eqref{bd9} hold. Let PCG be initialized with the zero solution and terminated according to criterion \eqref{bd59}. If $\eta^k$ in \eqref{bd59} satisfies $\lim_{k\to\infty}\eta^k = 0$,
then pdNCG converges superlinearly. 
\end{theorem}
\begin{proof}
Let $r_c^\mu(x^k) :=\nabla^2 f_c^\mu(x^k)\Delta x^k + \nabla f_c^\mu(x^k)$. We rewrite
$$
r_c^\mu(x^k)=\nabla^2 f_c^\mu(x^k)\Delta x^k  + \hat{B}^k\Delta x^k - \hat{B}^k\Delta x^k + \nabla f_c^\mu(x^k).
$$
Using Cauchy-Schwarz, \eqref{bd59} and Lemma \ref{lem:14} we have that for $k \to \infty$
\begin{align} \label{eq123}
\|r_c^\mu(x^k)\|_2 & \le \|\hat{B}^k\Delta x^k + \nabla f_c^\mu(x^k)\|_2 + \|\nabla^2 f_c^\mu(x^k)\Delta x^k - \hat{B}^k\Delta x^k\|_2 \nonumber \\
                              & \le \eta^k \|\nabla f_c^\mu(x^k)\|_2 +  \|\nabla f_c^\mu(x^k) - \hat{B}^k\|_2\|\Delta x^k\|_2 \nonumber \\
                              & = (\eta^k + \|\nabla f_c^\mu(x^k) - \hat{B}^k\|_2)\mathcal{O}(\|\nabla f_c^\mu(x^k)\|_2).
\end{align}

From Lemma \ref{lem:30} we have that $\hat{B}^k \to \nabla^2 f_c^\mu(x_{c,\mu})$ for $k\to \infty$.
and
from Theorem \ref{thm:5} we have that $\nabla^2 f_c^\mu(x^k) \to \nabla^2 f_c^\mu(x_{c,\mu})$ for $k\to \infty$.
Hence, if $\lim_{k\to\infty}\eta^k = 0$, then we have from \eqref{eq123}  that $\|r_c^\mu(x^k)\|_2=o(\|\nabla f_c^\mu(x^k)\|_2)$.
Moreover, from Lemma \ref{lem:4} we have that the Hessian of $f_c^\mu(x)$ is Lipschitz continuous. 
Therefore all conditions of part $(a)$ of Theorem $3.3$ in \cite{mybib:DES}
are satisfied, consequently pdNCG converges with superlinear rate of convergence. 
\end{proof}

\section{Preconditioning} \label{sec:cs}

Practical computational efficiency of pdNCG applied to system \eqref{eq108} depends on spectral properties of matrix $\hat{B}$ in \eqref{eq109}. Those can be improved by a suitable preconditioning. In this section we introduce a new preconditioner for $\hat{B}$ and discuss the limiting behaviour of the spectral properties of preconditioned $\hat{B}$.

First,  we give an intuitive analysis on the construction of 
the proposed preconditioner.
In Remark \ref{rem:3} it is mentioned that the distance $\omega$ of the two solutions $x_{c}$ and $x_{c,\mu} $ can be arbitrarily small for sufficiently small values of $\mu$.
Moreover, according to Assumption \ref{assum:1}, $W^*x_c$ is $q$ sparse. Therefore, Remark \ref{rem:3} implies that $W^*x_{c,\mu}$ is approximately $q$ sparse
with nearly zero components of $\mathcal{O}(\omega)$. A consequence of the previous 
statement is that the components of $W^*x_{c,\mu} $ split into the following disjoint sets
\begin{equation*}
\begin{aligned}
&\mathcal{B}  := \{i \in \{1,2,\cdots,l\}\ | \ |W^*_ix_{c,\mu} |\gg \mathcal{O}(\omega)\}, 	&   |\mathcal{B}|=q=|\mbox{supp}(W^*x_c)|,\\
& & \\
&\mathcal{B}^c  := \{i \in \{1,2,\cdots,l\}\ | \ |W^*_ix_{c,\mu} |\approx \mathcal{O}(\omega)\}, &    |\mathcal{B}^c|=l - q.
\end{aligned}
\end{equation*}
The behaviour of $W^*x_{c,\mu}$ has a crucial effect on matrix $\nabla^2 \psi_\mu(W^*x_{c,\mu})$ in \eqref{nabla2psi}.
Notice that the components of the diagonal matrix $D$, defined in \eqref{bd61} as part of $\nabla^2 \psi_\mu(W^*x_{c,\mu})$, split into two disjoint sets. 
In particular, $q$ components are non-zeros much less than $\mathcal{O}({1}/{\omega})$, while the majority, $l-q$, of its components are of $\mathcal{O}({1}/{\omega})$, 
\begin{equation}\label{eq:2}
D_i \ll \mathcal{O}(\frac{1}{\omega}) \quad \forall i\in\mathcal{B}  \quad \mbox{and} \quad D_i = \mathcal{O}(\frac{1}{\omega}) \quad \forall i\in\mathcal{B}^c.
\end{equation}
Hence, for points close to $x_{c,\mu}$ and small $\mu$, matrix $\nabla^2 f_c^\mu(x)$ in \eqref{eq131} consists of a dominant matrix $c\nabla^2 \psi_\mu(x) $ and
of matrix $A^\intercal A$ with moderate largest eigenvalue. The previous argument for $A^\intercal A$ is due to \eqref{bd8}. Observe that $\lambda_{max}(A^\intercal A)=\lambda_{max}(A A^\intercal)$,
hence, if $\delta$ in \eqref{bd8} is not a very large constant, then $\lambda_{max}(A^\intercal A) \le 1+\delta $.
According to Lemma \ref{lem:30}, the symmetric matrix $\mbox{sym}(\tilde{B})$ in \eqref{eq131} tends to matrix $\nabla^2 \psi_\mu(x)$ as $x \to x_{c,\mu}$. 
Therefore, matrix $\mbox{sym}(\tilde{B})$ is the dominant matrix in $\hat{B}$.
For this reason, in the proposed preconditioning technique, matrix $A^\intercal A$ in \eqref{eq131} is replaced by a scaled identity $\rho I_n$, $\rho>0$, while 
the dominant matrix $\mbox{sym}(\tilde{B})$ is maintained.
Based on these observations we propose the following preconditioner
\begin{equation}\label{bd10}
\tilde{N} := c\, \mbox{sym}(\tilde{B}) + \rho I_n.
\end{equation}
In order to capture the approximate separability of the diagonal components of matrix $D$ 
for points close to $x_{c,\mu}$, when $\mu$ is sufficiently small,
we will work with approximate guess of $\mathcal{B}$ and $\mathcal{B}^c$. 
For this reason, we introduce the positive constant $\nu$,
such that
$$
\# (D_i < \nu) = \sigma.
$$
Here $\sigma$ might be different from the sparsity of $W^*x_c$. Furthermore, according to the above definition we have the sets
\begin{equation}\label{eq:1}
\mathcal{B}_\nu := \{i \in \{1,2,\cdots,l\}\ | \ D_i < \nu \} \quad \mbox{and} \quad \mathcal{B}_\nu^c :=  \{1,2,\cdots,l\} \backslash \mathcal{B}_\nu,
\end{equation}
with $|\mathcal{B}_\nu|=\sigma$ and $|\mathcal{B}_\nu^c|=l-\sigma$. 
This notation is being used in the following theorem, in which
we analyze the behaviour of the spectral properties of preconditioned $\nabla^2 f_c^{\mu}(x)$, with preconditioner
$N:= c\nabla^2 \psi_\mu(W^*x) + \rho I_n$.
However, according to Lemma \ref{lem:30} matrices $\hat{B}$ and $\tilde{N}$ tend to $\nabla^2 f_c^\mu(x)$ and $N$, respectively, as $x\to x_{c,\mu}$.
Therefore, the following theorem is useful for
the analysis of the limiting behaviour of the spectral properties of preconditioned $\hat{B}$.
\begin{theorem}\label{thm:3}
Let $\nu$ be any positive constant and $\#(D_i < \nu) = \sigma$ at a point $x$, where $D$ is defined in \eqref{bd61}. 
Let 
$$
\nabla^2 f_c^\mu(x) = c\nabla^2 \psi_\mu(W^*x) + A^\intercal A \quad \mbox{and} \quad N:= c\nabla^2 \psi_\mu(W^*x) + \rho I_n.
$$
Additionally, let $A$ and $W$ satisfy W-RIP
with some constant $\delta_{\sigma} < 1/2$ and let $A$ satisfy \eqref{bd8} for some constant $\delta\ge0$.\\
If the eigenvectors of $N^{-\frac{1}{2}}\nabla^2 f_c^\mu(x) N^{-\frac{1}{2}} $ do not belong in $\mbox{Ker}(W^*_{\mathcal{B}_\nu^c})$ and $\rho\in[\delta_{\sigma},1/2]$, then the eigenvalues of $N^{-1}\nabla^2 f_c^\mu(x)$ satisfy
$$
|\lambda -1| \le \frac{1}{2}\frac{{\chi + 1} + (5\chi^2 - 2\chi + 1)^{\frac{1}{2}}}{ c\mu^2\nu^3\lambda_{min}(\mbox{Re}(W_{\mathcal{B}_\nu^c}W^*_{\mathcal{B}_\nu^c}))  + \rho},
$$
where $\lambda\in\mbox{spec}(N^{-1}\nabla^2 f_c^\mu(x) )$, $\lambda_{min}(\mbox{Re}(W_{\mathcal{B}_\nu^c}W^*_{\mathcal{B}_\nu^c}))$ is the minimum nonzero eigenvalue of $\mbox{Re}(W_{\mathcal{B}_\nu^c}W^*_{\mathcal{B}_\nu^c})$
and  $\chi:= 1+\delta - \rho$. \\
If the eigenvectors of $N^{-\frac{1}{2}}\nabla^2 f_c^\mu(x) N^{-\frac{1}{2}} $ belong in $\mbox{Ker}(W^*_{\mathcal{B}_\nu^c})$, then
$$
|\lambda -1| \le \frac{1}{2}\frac{{\chi + 1} + (5\chi^2 - 2\chi + 1)^{\frac{1}{2}}}{\rho}.
$$
\end{theorem}
\begin{proof}
We analyze the spectrum of matrix $N^{-\frac{1}{2}}\nabla^2 f_c^\mu(x) N^{-\frac{1}{2}}$ instead, because it has the same eigenvalues
as matrix $N^{-1} \nabla^2 f_c^\mu(x)$. We have that
\begin{align*}
N^{-\frac{1}{2}}\nabla^2 f_c^\mu(x) N^{-\frac{1}{2}} & =  N^{-\frac{1}{2}}(c\nabla^2\psi_{\mu}(x) + A^\intercal A ) N^{-\frac{1}{2}} \\  
                                                                     & =  N^{-\frac{1}{2}}(c\nabla^2\psi_{\mu}(x) + A^\intercal A  + \rho I_n - \rho I_n) N^{-\frac{1}{2}} \\
                                                                     & =  N^{-\frac{1}{2}}(c\nabla^2\psi_{\mu}(x) + \rho I_n ) N^{-\frac{1}{2}} +  N^{-\frac{1}{2}} A^\intercal A N^{-\frac{1}{2}} - \rho N^{-1}\\
                                                                     & = I_n  + N^{-\frac{1}{2}} A^\intercal A N^{-\frac{1}{2}} - \rho N^{-1}
\end{align*}
Let $u$ be an eigenvector of $N^{-\frac{1}{2}}\nabla^2 f_c^\mu(x) N^{-\frac{1}{2}} $ with $\|u\|_2=1$ and $\lambda$ the corresponding eigenvalue, then 
\begin{align}\label{eq:501}
(I_n + N^{-\frac{1}{2}} A^\intercal A N^{-\frac{1}{2}} - \rho N^{-1})u & = \lambda u& \Longleftrightarrow \nonumber \\
(N + N^{\frac{1}{2}} A^\intercal A N^{-\frac{1}{2}} - \rho I_n )u          & = \lambda N u& \Longrightarrow \nonumber \\
u^\intercal N^{\frac{1}{2}} (A^\intercal A - \rho I_n  )N^{-\frac{1}{2}}  u          & = (\lambda -1)u^\intercal N u& \Longrightarrow \nonumber \\
|u^\intercal N^{\frac{1}{2}} (A^\intercal A   - \rho I_n  )N^{-\frac{1}{2}}u|          & = |\lambda -1|u^\intercal N u.&
\end{align}
First, we find an upper bound for $|u^\intercal N^{\frac{1}{2}} (A^\intercal A   - \rho I_n  )N^{-\frac{1}{2}}u|$. Matrices $N^{\frac{1}{2}} (A^\intercal A   - \rho I_n  )N^{-\frac{1}{2}}$
and $A^\intercal A   - \rho I_n$ have the same eigenvalues. Therefore, 
\begin{align*}
|u^\intercal N^{\frac{1}{2}} (A^\intercal A   - \rho I_n  )N^{-\frac{1}{2}}u| \le \lambda_{max}^{+}(A^\intercal A - \rho I_n)
\end{align*}
where $\lambda_{max}^{+}(\cdot)$ is the largest eigenvalue of the input matrix in absolute value. Thus,
\begin{align*}
|u^\intercal N^{\frac{1}{2}} (A^\intercal A   - \rho I_n  )N^{-\frac{1}{2}}u|  &\le \mymax_{\|v\|_2^2\le 1} |v^\intercal(A^\intercal A - \rho I_n) v | \\
          												  & =  \mymax_{\|Pv\|_2^2 + \|Qv\|_2^2\le 1} |(Pv + Qv)^\intercal(A^\intercal A - \rho I_n) (Pv + Qv) |, 
\end{align*}
where $P$ is the projection matrix to the column space of $W_{\mathcal{B}_\nu}$ and $Q=I_n - P$. Using triangular inequality we get
\begin{align*}
|u^\intercal N^{\frac{1}{2}} (A^\intercal A   - \rho I_n  )N^{-\frac{1}{2}}u|  & \le  \mymax_{\|Pv\|_2^2 + \|Qv\|_2^2\le 1} \big(|(Pv)^\intercal (A^\intercal A - \rho I_n)Pv|  \nonumber \\
													   & + |(Qv)^\intercal (A^\intercal A - \rho I_n)Qv| 
 													      + 2 |(Pv)^\intercal (A^\intercal A - \rho I_n)Qv|\big).
\end{align*}
Let us denote by $\hat v$ the solution of this maximization problem and set $\|P \hat v\|_2^2=\alpha$ and $\|Q \hat v\|_2^2=1-\alpha$, where $\alpha \in [0,1]$, then
\begin{align}\label{eq:152}
|u^\intercal N^{\frac{1}{2}} (A^\intercal A   - \rho I_n  )N^{-\frac{1}{2}}u|  & \le \big(|(P \hat v)^\intercal (A^\intercal A - \rho I_n)P \hat v|  \nonumber \\
													   & + |(Q \hat v)^\intercal (A^\intercal A - \rho I_n)Q \hat v| 
 													      + 2 |(P \hat v)^\intercal (A^\intercal A - \rho I_n)Q \hat v|\big).
\end{align}
Since $P \hat v$ belongs to the column space of $W_{\mathcal{B}_\nu}$ and $|\mathcal{B}_\nu|=\sigma$, from W-RIP with $\delta_\sigma < 1/2$ we have that
\begin{align*}
\|P\hat v\|_2^2 (1-\delta_\sigma) & \le \|AP \hat v\|_2^2 & \Longrightarrow \\
\|P\hat v\|_2^2 (1-\rho) &\le \|AP\hat v\|_2^2 & \Longleftrightarrow \\
\|P\hat v\|_2^2 (1-2\rho) &\le \|AP\hat v\|_2^2 - \rho \|P\hat v\|_2^2.  & 
\end{align*}
Since $\rho\in [\delta_\sigma, 1/2]$ we have that 
$
\rho \|P\hat v\|_2^2 \le \|AP\hat v\|_2^2,
$
which implies that if the eigenvector corresponding to an eigenvalue of matrix $A^\intercal A$ belongs to the column space of $W_{\mathcal{B}_\nu}$, then
the eigenvalue cannot be smaller than $\rho$.
Hence,
$$
 |(P\hat v)^\intercal (A^\intercal A - \rho I_n)P\hat v| \le |(P\hat v)^* (A^\intercal A - \rho I_n)P\hat v| = (P\hat v)^* (A^\intercal A - \rho I_n)P\hat v.
$$
 Moreover, from W-RIP with $\delta_\sigma<1/2$ and $\rho\in [\delta_\sigma, 1/2]$, we also have that $(P\hat v)^* (A^\intercal A - \rho I_n)P\hat v \le \|P\hat v\|_2^2$. 
 Thus,
  \begin{equation}\label{eq:1000}
 |(P\hat v)^\intercal (A^\intercal A - \rho I_n)P\hat v| \le \alpha.
 \end{equation}
 From property \eqref{bd8} and $\lambda_{max}(A^\intercal A)=\lambda_{max}(A A^\intercal)$, we have that 
 $\lambda_{max}(A^\intercal A - \rho I_n) \le 1+\delta - \rho$. Finally, using the Cauchy-Schwarz inequality, we get that
 \begin{equation}\label{eq:1001}
|(Q\hat v)^\intercal (A^\intercal A - \rho I_n)Q\hat v| \le (1+\delta -\rho)(1-\alpha)
\end{equation}
and
 \begin{equation}\label{eq:1002}
 |(P \hat v)^\intercal (A^\intercal A - \rho I_n)Q\hat v|\le (1+\delta - \rho)\sqrt{\alpha(1-\alpha)}.
 \end{equation}
Using \eqref{eq:1000}, \eqref{eq:1001} and \eqref{eq:1002} in \eqref{eq:152} we have that 
\begin{align} \label{eq:153}
|u^\intercal N^{\frac{1}{2}} (A^\intercal A   - \rho I_n  )N^{-\frac{1}{2}}u|  
													   & \le \alpha + (1+\delta - \rho)(1-\alpha) + 2(1+\delta - \rho)\sqrt{\alpha(1-\alpha)}. 
\end{align}
Set 
$\chi:= 1+\delta - \rho$,
it is easy to check that in the interval $\alpha \in [0,1]$ the right hand side of \eqref{eq:153} has a maximum at one of the four candidate points
$$
\alpha_1 = 0, \quad \alpha_2 = 1, \quad \alpha_{3,4} = \frac{1}{2}(1 \pm \left(\frac{(\chi - 1)^2}{5\chi^2 - 2\chi + 1}\right)^{1/2}),
$$
where $\alpha_3$ is for plus and $\alpha_4$ is for minus. The corresponding function values are
$$
\chi, \quad 1, \quad \frac{\chi + 1}{2} + \frac{1}{2}\frac{3\chi^2 + 2\chi -1}{(5\chi^2 - 2\chi + 1)^{{1}/{2}}}, \quad  \frac{\chi + 1}{2} + \frac{1}{2}(5\chi^2 - 2\chi +1 )^{{1}/{2}},
$$
respectively.
Hence, the maximum among these four values is given for $\alpha_4$.
Thus, \eqref{eq:153} is upper bounded by
\begin{equation}\label{eq:154}
|u^\intercal N^{\frac{1}{2}} (A^\intercal A   - \rho I_n  )N^{-\frac{1}{2}}u|  \le \frac{\chi + 1}{2} + \frac{1}{2}(5\chi^2 - 2\chi + 1)^{\frac{1}{2}}.
\end{equation}
We now find a lower bound for $u^\intercal N u$. 
Using the definition of $D$ in \eqref{bd61}, matrix $\hat Y$ in \eqref{bd62} is rewritten as
$
\hat Y_i = (2\mu^2 + |y_i|^2) D_{i}^3\   \forall i=1,2,\cdots,l.
$
Thus $\nabla^2 \psi_\mu (x)$ in \eqref{nabla2psi} is rewritten as
\begin{align}\label{eq120}
\nabla^2\psi_{\mu}(W^*x) & =\frac{1}{4}[(W \tilde{D}^3 W^* + \bar W \tilde{D}^3 \bar W^* + W \tilde{Y} \bar W^* + \bar W \tilde{\bar Y} W^*) \\ 
                                         & + 2\mu^2 (W D^3 W^* + \bar W D^3 \bar W^*)],\nonumber
\end{align}
where $\tilde{D}_{i} =|y_i|^2 D_{i}^3$ $\forall i=1,2,\cdots,l$.
Observe, that matrix $\nabla^2\psi_{\mu}(W^*x)$ consists of two matrices $W \tilde{D}^3 W^* + \bar W \tilde{D}^3 \bar W^* + W \tilde{Y} \bar W^* + \bar W \tilde{\bar Y} W^*$ and 
$2\mu^2 (W D^3 W^* + \bar W D^3 \bar W^*)$ which are positive semi-definite. 
Using \eqref{eq120} and the previous statement 
we get that 
\begin{align*}
u^\intercal N u & = u^\intercal (c\nabla^2\psi_\mu(W^*x) + \rho I_n) u \\
                        & = \frac{c}{4}u^\intercal(W \hat{Y} W^* + \bar W \hat{ {Y}} \bar W^* + W \tilde{Y} \bar W^* + \bar W \tilde{\bar Y} W^* )u +\rho\\
                        & \ge \frac{c\mu^2}{2}u^\intercal (WD^3W^* + \bar{W}D^3\bar{W}^*)u +\rho.
\end{align*}
Furthermore, using the splitting of matrix $D$ \eqref{eq:1}, the last inequality is equivalent to
\begin{align*}
u^\intercal N u & = \frac{c\mu^2}{2}u^\intercal (W_{\mathcal{B}_\nu}D^3_{\mathcal{B}_\nu} W^*_{\mathcal{B}_\nu} + W_{\mathcal{B}_\nu^c}D^3_{\mathcal{B}_\nu^c} W^*_{\mathcal{B}_\nu^c} + 
                        \bar{W}_{\mathcal{B}_\nu}D^3_{\mathcal{B}_\nu}\bar{W}^*_{\mathcal{B}_\nu} + \bar{W}_{\mathcal{B}_\nu^c}D^3_{\mathcal{B}_\nu^c}\bar{W}^*_{\mathcal{B}_\nu^c})u   +\rho  \\
                        & \ge \frac{c\mu^2}{2}u^\intercal(W_{\mathcal{B}_\nu^c}D^3_{\mathcal{B}_\nu^c} W^*_{\mathcal{B}_\nu^c}  + \bar{W}_{\mathcal{B}_\nu^c}D^3_{\mathcal{B}_\nu^c}\bar{W}^*_{\mathcal{B}_\nu^c})u  + \rho. 
\end{align*}
Using the defition of $\mathcal{B}_\nu^c$ \eqref{eq:1} in the last inequality, the quantity $u^\intercal N u $ is further lower bounded by
\begin{equation}\label{eq:150}
                     u^\intercal N u    \ge \frac{c\mu^2\nu^3}{2}u^\intercal(W_{\mathcal{B}_\nu^c}W^*_{\mathcal{B}_\nu^c}  + \bar{W}_{\mathcal{B}_\nu^c}\bar{W}^*_{\mathcal{B}_\nu^c})u+\rho.
\end{equation}
If $u\notin \mbox{Ker}(W^*_{\mathcal{B}_\nu^c})$, then from \eqref{eq:150} we get
\begin{equation} \label{eq:151}
u^\intercal N u \ge  c\mu^2\nu^3\lambda_{min}(\mbox{Re}(W_{\mathcal{B}_\nu^c}W^*_{\mathcal{B}_\nu^c}))  + \rho. 
\end{equation}
Hence, combining \eqref{eq:501}, \eqref{eq:154} and \eqref{eq:151} we conclude that 
\begin{align*}
|\lambda-1|  \le \frac{1}{2}\frac{{\chi + 1} + (5\chi^2 - 2\chi + 1)^{\frac{1}{2}}}{ c\mu^2\nu^3\lambda_{min}(\mbox{Re}(W_{\mathcal{B}_\nu^c}W^*_{\mathcal{B}_\nu^c}))  + \rho} 
\end{align*}
If $u\in \mbox{Ker}(W^*_{\mathcal{B}_\nu^c})$, then from \eqref{eq:150} we have that $u^\intercal N u \ge \rho$, hence 
$$
|\lambda -1| \le \frac{1}{2}\frac{{\chi + 1} + (5\chi^2 - 2\chi + 1)^{\frac{1}{2}}}{\rho}.
$$
\end{proof}

Let us now draw some conclusions from Theorem \ref{thm:3}. In order for the eigenvalues of $N^{-1}\nabla^2 f_c^\mu(x)$ to be around one, it is required that the degree of freedom $\nu$ is chosen such that $\nu = \mathcal{O}(1/\mu)$
and $\mu$ is small.
For such $\nu$, the cardinality $\sigma$ of the set $\mathcal{B}_\nu$ must be small enough such that 
matrices $A$ and $W$ satisfy W-RIP with constant $\delta_{\sigma}< 1/2$; otherwise the assumptions of Theorem \ref{thm:3} will not be satisfied. 
This is possible if the pdNCG iterates are close to the optimal solution $x_{c,\mu}$ and $\mu$ is sufficiently small. In particular, for  sufficiently small $\mu$, 
from Remark \ref{rem:3} we have that $x_{c,\mu} \approx x_c$ and $\sigma \approx q$. 
According to Assumption \ref{assum:1} for the $q$-sparse $x_c$, W-RIP is satisfied for $\delta_{2q} < 1/2 \Longrightarrow \delta_q < 1/2$.
Hence, for points close to $x_{c,\mu}$ and small $\mu$ we expect that $\delta_\sigma < 1/2$.
Therefore, the result in Theorem \ref{thm:3} captures only the limiting behaviour of preconditioned $\nabla^2 f_c^\mu(x)$ as $x\to x_{c,\mu}$.
Moreover, according to Lemma \ref{lem:30}, Theorem \ref{thm:3}
implies that at the limit the eigenvalues of $\tilde{N}^{-1}\hat{B}$ are also clustered around one.
However, the scenario of limiting behaviour of the preconditioner is pessimistic.
Let $\tilde{\sigma}$ be the minimum sparsity level such that matrices $A$ and $W$ are W-RIP with $\delta_{\tilde{\sigma}} < 1/2$. Then,
according to the uniform property of W-RIP (i.e. it holds for all at most $\tilde{\sigma}$-sparse vectors), 
the preconditioner will start to be effective even if the iterates $W^*x^k$ are approximately sparse with $\tilde{\sigma}$ dominant non-zero components.
Numerical evidence is provided in Figure \ref{figSpec} which verifies the previous.
In Figure \ref{figSpec} the spectra $\lambda(\hat{B})$ and $\lambda(\tilde{N}^{-1}\hat{B})$ are displayed for a sequence of systems which arise 
when an iTV problem is solved. For this iTV problem we set matrix $A$ to be a partial $2D$ DCT, $n=2^{10}$, $m=n/4$, $c=2.29e$-$2$
and $\rho=5.0e$-$1$. 
For the experiment in Figures \ref{fig5_1_a} and \ref{fig5_1_b} the smoothing parameter has been set to $\mu=1.0e$-$3$
and in Figures \ref{fig5_1_c} and \ref{fig5_1_d} $\mu=1.0e$-$5$. Observe that for both cases the spectrum of matrix $\tilde{N}^{-1}\hat{B}$ is substantially restrained around one in comparison to the spectrum of matrix $\hat{B}$
which has large variations. Notice that the preconditioner was effective not only at optimality as it was predicted by theory, but through all iterations of pdNCG. This is because starting from the zero solution 
the iterates $W^*x^k$ were maintained approximately sparse $\forall k$. 
\begin{figure}
\centering
	\begin{subfigure}[b]{0.48\textwidth}
		\includegraphics[width=\textwidth]{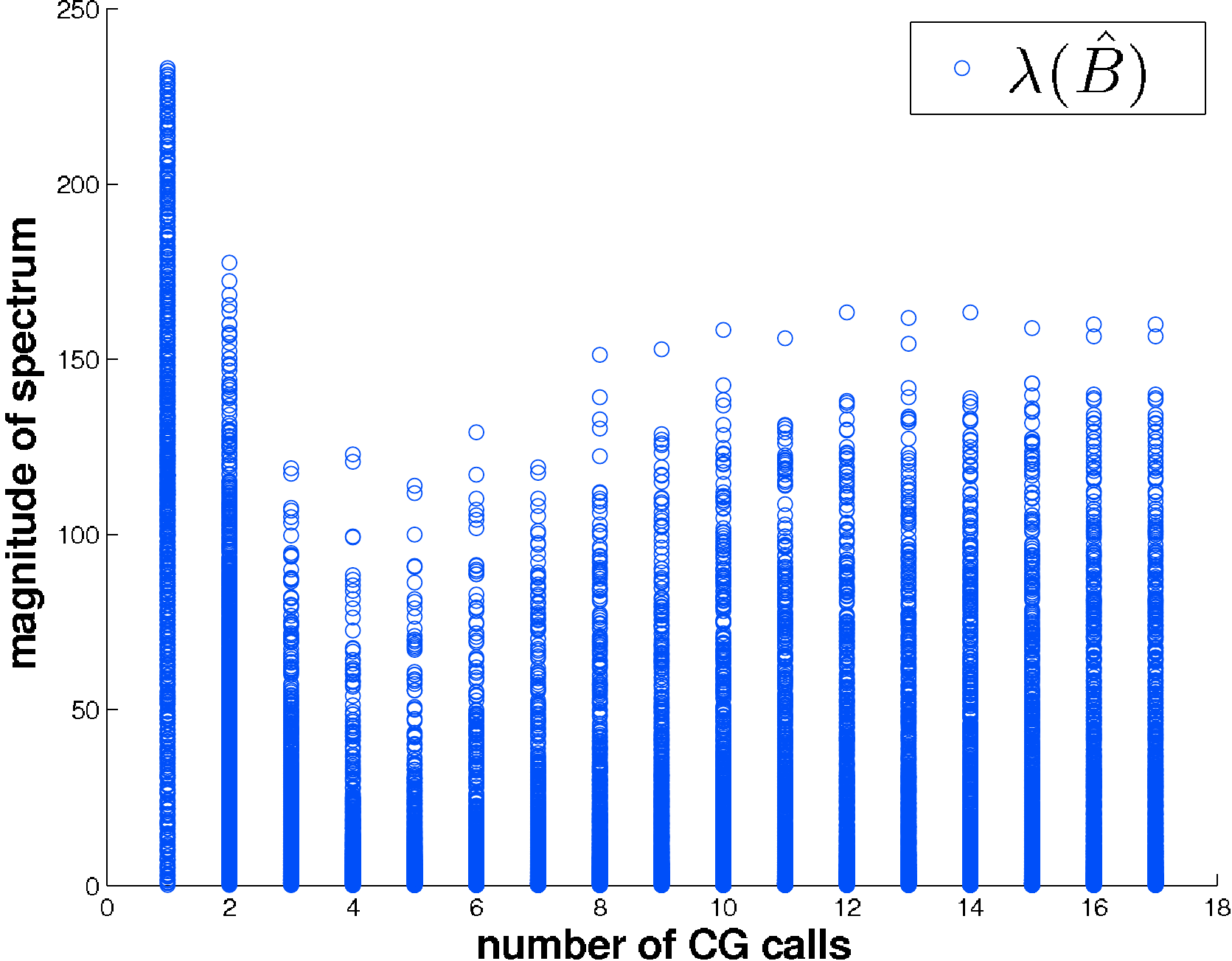}		
		\caption{Unpreconditioned}
		\label{fig5_1_a}%
         \end{subfigure}
         \quad
	\begin{subfigure}[b]{0.48\textwidth}
		\includegraphics[width=\textwidth]{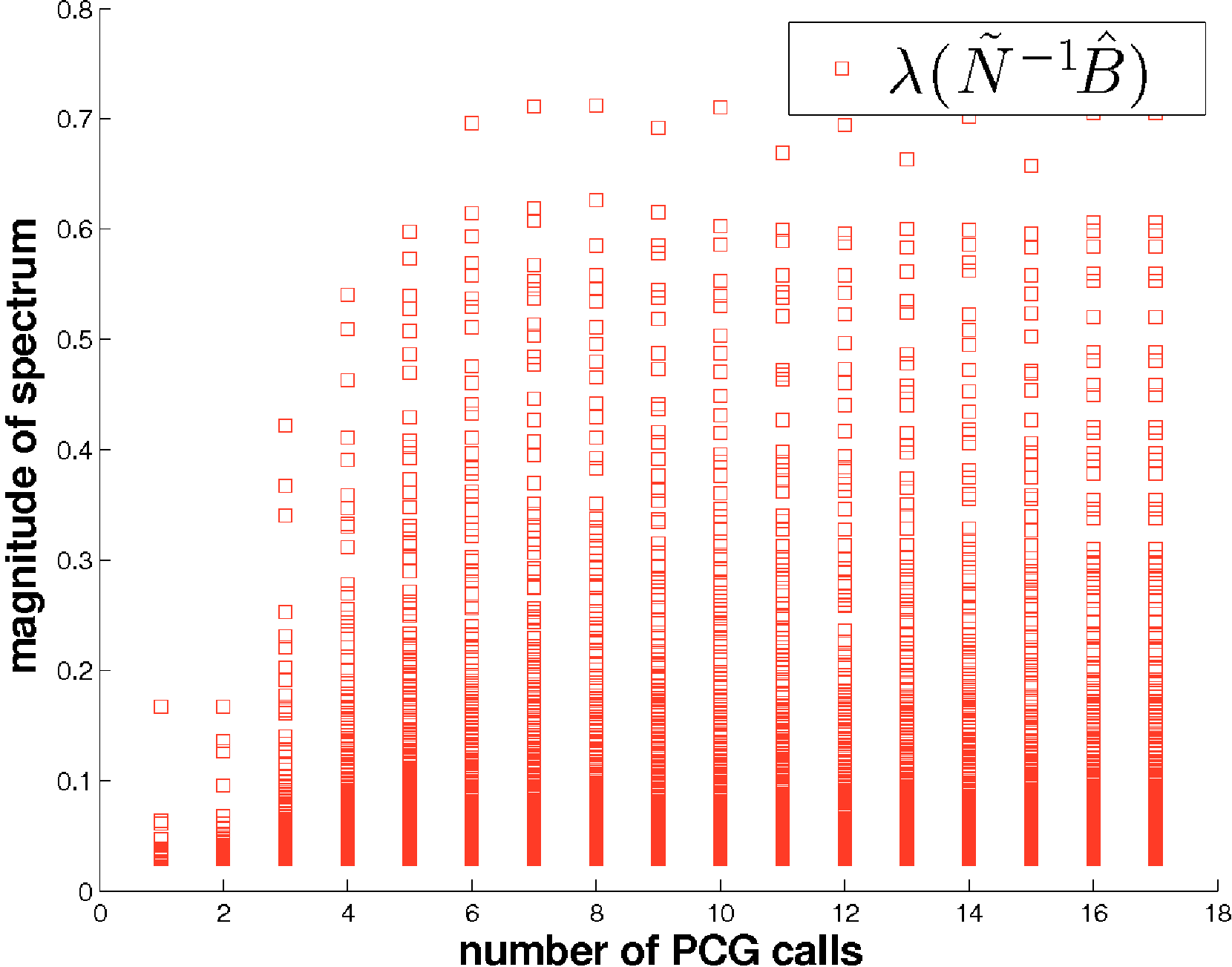}	
		\caption{Preconditioned}
		\label{fig5_1_b}%
         \end{subfigure}
	\begin{subfigure}[b]{0.48\textwidth}
		\includegraphics[width=\textwidth]{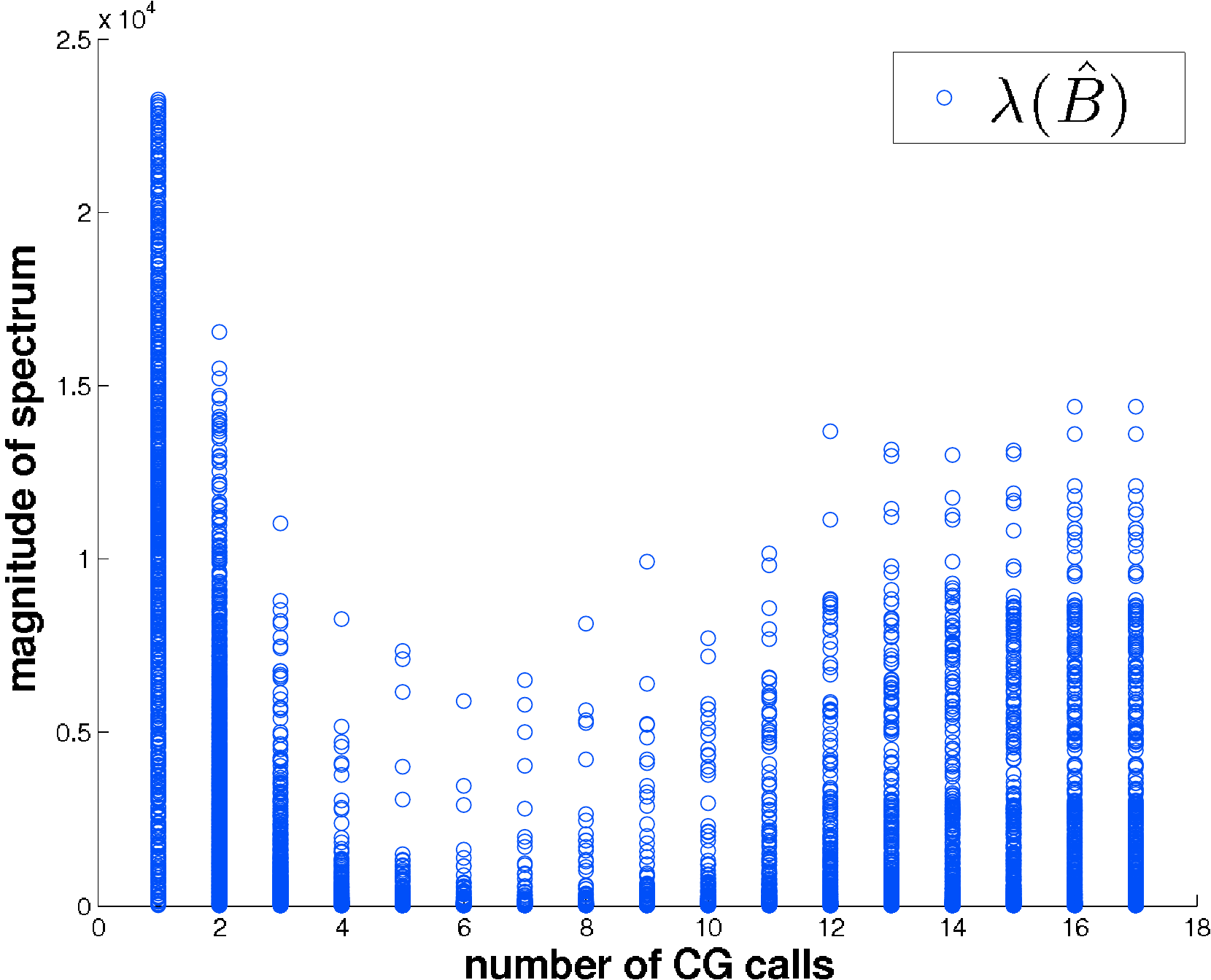}		
		\caption{Unpreconditioned}
		\label{fig5_1_c}%
         \end{subfigure}
         \quad
	\begin{subfigure}[b]{0.48\textwidth}
		\includegraphics[width=\textwidth]{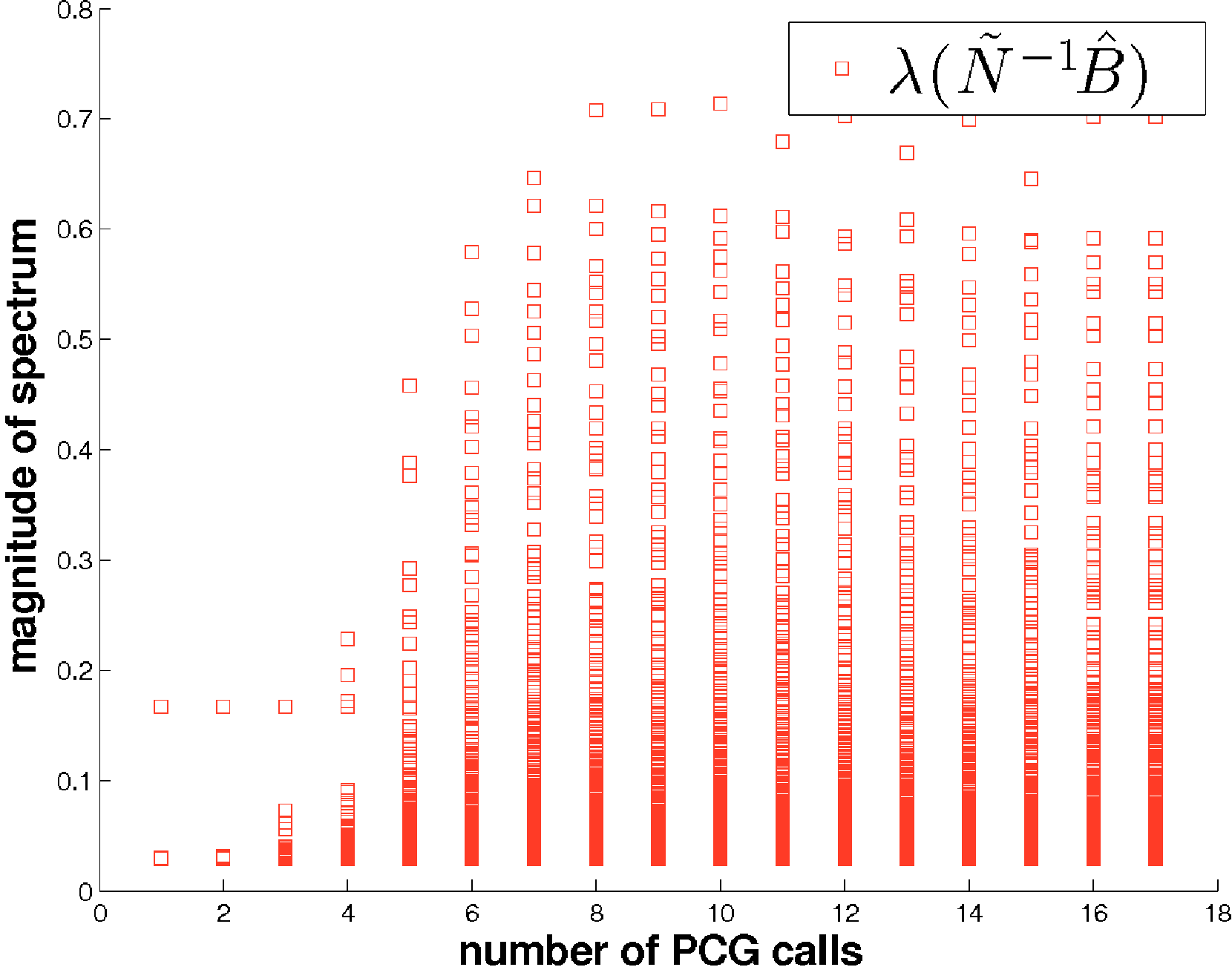}	
		\caption{Preconditioned}
		\label{fig5_1_d}%
         \end{subfigure}
	\caption{Spectra of $\lambda(\hat{B})$ and $\lambda(\tilde{N}^{-1}\hat{B})$ when pdNCG is applied with smoothing parameter $\mu=1.0e$-$3$ (top sub-figures)
	and $\mu=1.0e$-$5$ (bottom sub-figures). 
	Matrix $A$ in $\hat{B}$ is a $2D$ DCT, $n=2^{10}$, $m=n/4$ and $c=2.29e$-$2$. Seventeen systems are solved in total for each experiment.
	}
	\label{figSpec}%
\end{figure}

We now comment on the second result of Theorem \ref{thm:3}, when the eigenvectors of $N^{-\frac{1}{2}}\nabla^2 f_c^\mu(x) N^{-\frac{1}{2}} $ belong in $\mbox{Ker}(W^*_{\mathcal{B}_\nu^c})$.
In this case, according to Theorem \ref{thm:3} the preconditioner removes the disadvantageous dependence of the spectrum of $\nabla^2 f_c^\mu(x)$ on the smoothing parameter $\mu$.
However, there is no guarantee that the eigenvalues of $N^{-1}\nabla^2 f_c^\mu(x)$ are clustered around one, regardless of the distance from the optimal solution $x_{c,\mu}$. 
Again, because of Lemma \ref{lem:30} we expect that the spectrum 
of $\tilde{N}^{-1}\hat{B}$ at the limit will have a similar behaviour. 

Finally, a question arises regarding the computational cost of solving systems with preconditioner $\tilde{N}$; it is necessary that this operation is inexpensive. 
For iTV problems $W$ is a \textit{five-diagonal banded matrix}, which arises from discretization of the nabla operator applied on an image; see section $4.5$ in \cite{convexTemplates} for details about matrix $W$ for iTV problems.
From \eqref{eq121} we deduce that matrix $\tilde{N}$ is also five-diagonal banded matrix. Hence $\tilde{N}$ can be computed inexpensively and systems with it can be solved exactly and fast, by using specialized solvers for banded matrices.
Unfortunately, for $\ell_1$-analysis matrix $W$ has no structure. Therefore $\tilde{N}$ might be expensive to compute or store in memory. 
However, it is common that algorithms are available for fast matrix-vector products with matrices $W$ and $W^\intercal$. In this case, systems with matrix $\tilde{N}$ can be solved approximately using CG. 
The idea of approximate preconditioning has also been employed successfully for specialized image reconstruction problems, i.e. denoising, in \cite{approxprec}. 
In Figure \ref{fig8} we present the performance of preconditioner $\tilde{N}$  when it is used in an approximate setting. The tested problem is the same iTV problem which was described previously for Figures \ref{figSpec}, but the size of the problem
is changed to $n=2^{16}$ and $\mu=1.0e$-$5$. For the approximate solution of systems with the preconditioner $\tilde{N}$ we required from CG to perform $15$ iterations and then the process was truncated. Observe in Figure \ref{fig8} that this resulted in a substantial reduction in the number of PCG iterations
compared to the unpreconditioned case. 
Additionally, the approximate preconditioner setting was $1.4$ times faster in terms of the overall CPU time for convergence. 

\begin{figure}
\centering
	\begin{subfigure}[b]{0.7\textwidth}
		\includegraphics[width=\textwidth]{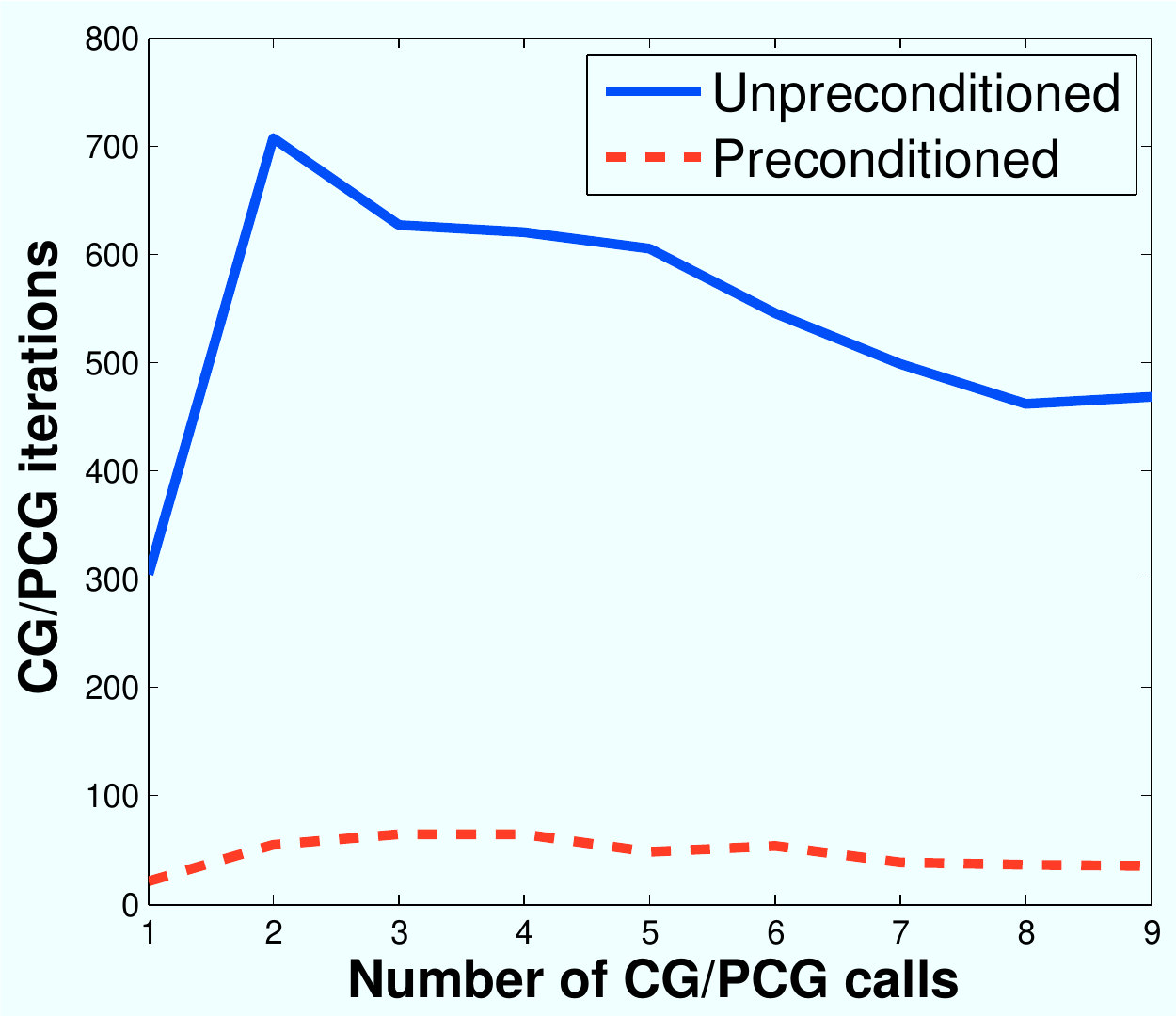}		
         \end{subfigure}
         \caption{Significant reduction in the number of PCG iterations when preconditioner $\tilde{N}$ is applied in an approximate setting. 
         For this example, matrix $A$ in $\hat{B}$ is a $2D$ DCT, $n=2^{16}$, $m=n/4$, $c=5.39e$-$2$ and $\mu=1.0e$-$5$.
         }
	\label{fig8}%
\end{figure}

\section{Continuation}\label{sec:cont}
In the previous section we have shown that by using preconditioning, the spectral properties of systems which arise can be improved. However, for initial stages of pdNCG
a similar result can be achieved without the cost of having to apply preconditioning. In particular, at initial stages 
the spectrum of $\hat{B}$ can be controlled to some extent through inexpensive continuation. Whilst preconditioning is enabled only at later stages of the process. 
Briefly by continuation it is meant that a sequence of ``easier" subproblems are solved, instead of solving directly problem \eqref{prob2}. 
The reader is referred to Chapter $11$ in \cite{mybib:NocedalWright} for a survey on continuation methods in optimization.

 In this paper we use a similar continuation framework to \cite{ctnewtonold,ctpdnewton}. In particular, 
 a sequence of sub-problems \eqref{prob2} are solved, where each of them is parameterized by $c$ and $\mu$ simultaneously.
 Let $\tilde{c}$ and $\tilde{\mu}$ be the final parameters for which problem \eqref{prob2} must be solved. Then
 the number of continuation iterations $\vartheta$ is set to be the maximum order of magnitude between $1/\tilde{c}$ and $1/\tilde{\mu}$.
 For instance, if $\tilde{c}=1.0e$-$2$ and $\tilde{\mu}=1.0e$-$5$ then $\vartheta := \max(2,5) = 5$.
 If $\vartheta \ge 2$, then the initial parameters $c^0$ and $\mu^0$ are both always set to $1.0e$-$1$ and the intervals $[c^0,\tilde{c}]$ and $[\mu^0,\tilde{\mu}]$
 are divided in $\vartheta$ equal subintervals in logarithmic scale. For all experiments that we have performed in this paper we have found that 
 this setting leads to a generally acceptable improvement over pdNCG without continuation.
 The pseudo-code of the proposed continuation framework is shown in Figure \ref{fig:1}.
 
 \begin{figure}
\begin{algorithmic}[1]
\vspace{0.1cm}
\STATE \textbf{Outer loop:} For $j=0,1,2,\ldots ,\vartheta$, produce $(c^j,\mu^j)_{j=0}^\vartheta$.
\STATE \hspace{0.5cm}\textbf{Inner loop:} Approximately solve the subproblem
\begin{equation*}
 \mbox{minimize} \  f_{c^j}^{\mu^j}(x)
\end{equation*}
\hspace{0.5cm} using pdNCG and by initializing it with the  solution\\ 
\hspace{0.5cm} of the previous subproblem.
\end{algorithmic}
\caption{Continuation framework}
\label{fig:1}
\end{figure}

Figure \ref{fig3} shows the performance of pdNCG for three cases, no continuation with preconditioning, continuation with preconditioning through the whole process 
and continuation with preconditioning only at later stages. The vertical axis of Figure \ref{fig3} shows the relative error $\|x^k-x_{\tilde{c},\tilde{\mu}}\|_2/\|x_{\tilde{c},\tilde{\mu}}\|_2$.
The optimal $x_{\tilde{c},\tilde{\mu}}$ is obtained by using pdNCG with parameter tuning set to recover a highly accurate solution. The horizontal axis shows the CPU time.
The problem is an iTV problem were matrix $A$ is a partial $2D$ DCT, $n=2^{16}$, $m=n/4$, $c=5.39e$-$2$
and $\rho=5.0e$-$1$. The final smoothing parameter $\tilde{\mu}$ is set to $1.0e$-$5$. For the experiment that preconditioning is used only at later stages of continuation;
preconditioning is enabled when $\mu^j \le 1.0e$-$4$, where $j$ is the counter for continuation iterations. 
All experiments are terminated when the relative error $\|x^k-x_{\tilde{c},\tilde{\mu}}\|_2/\|x_{\tilde{c},\tilde{\mu}}\|_2 \le 1.0e$-$1$. Solving approximately the problem is an acceptable
practise since the problem is very noisy (i.e. signal-to-noise-ratio is $10$ dB) and there is not much improvement of the reconstructed image if more accurate solutions are requested.
Finally, all other parameters of pdNCG were set to the same values for all three experiments.
Observe in Figure \ref{fig3} that continuation with preconditioning only at late stages was the best approach for this problem. 

\begin{figure}
\centering
	\begin{subfigure}[b]{0.7\textwidth}
		\includegraphics[width=\textwidth]{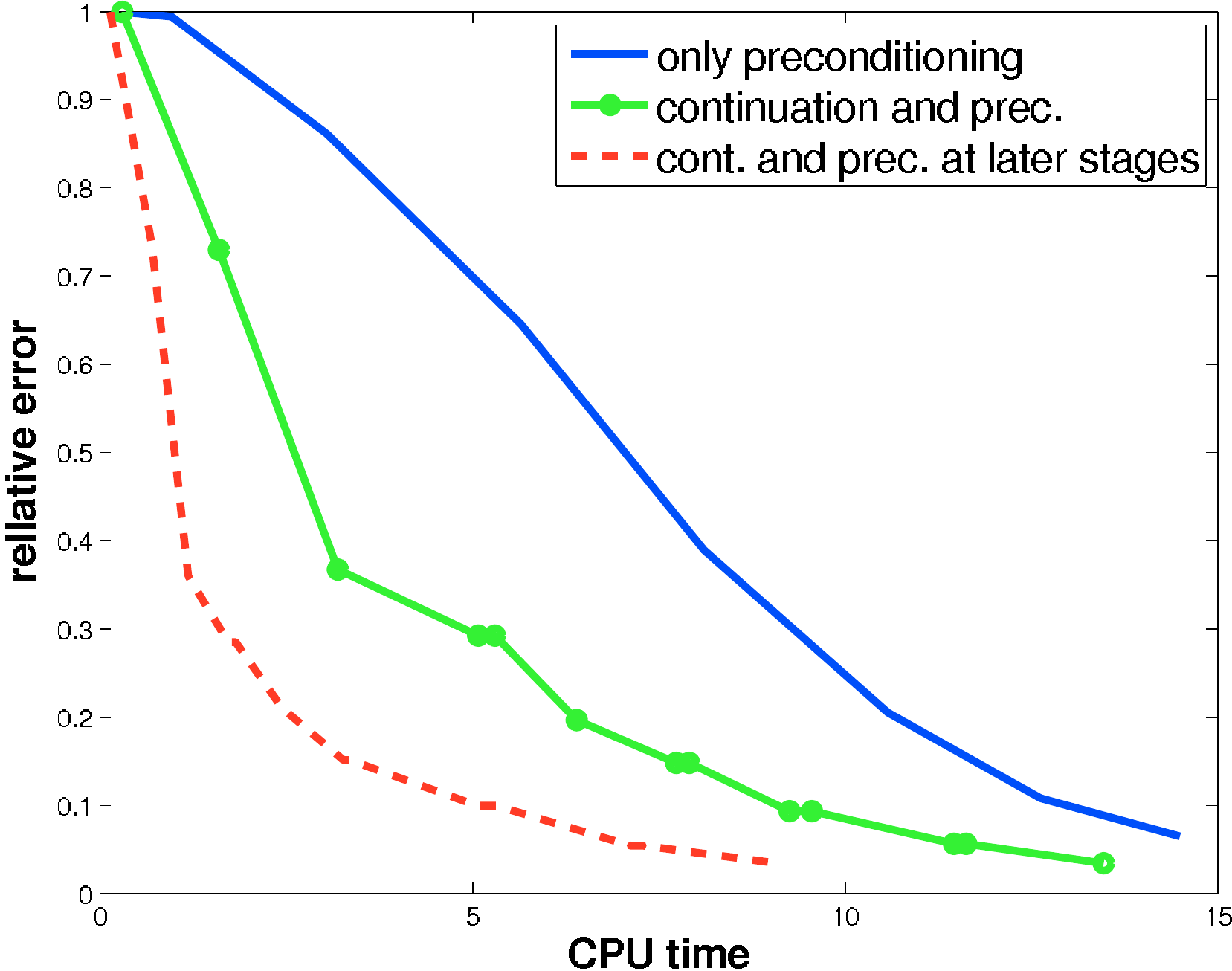}		
         \end{subfigure}
         \caption{Performance of pdNCG for three different settings, i) no continuation with preconditioning, ii) continuation with preconditioning through all iterations and iii) continuation 
         with preconditioning only at later stages. The vertical axis presents the relative error $\|x^k-x_{\tilde{c},\tilde{\mu}}\|_2/\|x_{\tilde{c},\tilde{\mu}}\|_2$, where $x_{\tilde{c},\tilde{\mu}}$ is 
         the optimal solution for the parameter setting $\tilde{c}$, $\tilde{\mu}$ in problem \eqref{prob2}.}
	\label{fig3}%
\end{figure}

\section{Numerical Experience}\label{secNumExp}

In this section we demonstrate the efficiency of pdNCG against a state-of-the-art method on  $\ell_1$-analysis with coherent and redundant dictionaries and iTV problems. 
In what follows we briefly discuss existing methods, we describe the setting of the experiments and finally numerical
results are presented. All experiments that are demonstrated in this paper can be reproduce by downloading the software from 
\url{http://www.maths.ed.ac.uk/ERGO/pdNCG/}.

\subsection{Existing Algorithms}\label{sec:algos}

Although the field of developing algorithms for iTV problems such as image denoising, deblurring and in-painting is densely populated, this is not the case for general CS problems with coherent and redundant dictionaries. 
For example, the solvers
NestA  and C-SALSA \cite{csalsa} can also solve \eqref{prob1} but they are applicable only in the case that $A A^\intercal = I$. Moreover, the solver Generalized Iterative Soft Thresholding (GISTA) in \cite{gista} 
requires that $\|A\|_2\le \sqrt{2}$ and $\|W^*\|_2\le 1$. This can be achieved by appropriate scaling of matrices $A$ and $W$, however, one needs to know a-priori an estimation of $\|A\|_2$ and $\|W^*\|_2$,
which might not be possible in practise. Another state-of-the-art method is the Primal-Dual Hybrid Gradient (PDHG) in \cite{pdhg}.
For this method no requirements are needed for matrices $A$ and $W$. 
PDHG has been reported to be very efficient for imaging applications such as denoising and deblurring, for which
matrix $A$ is the identity or a square and full-rank matrix which is inexpensively diagonalizable. Unfortunately, this is not always the case for the CS problems which we are interested in. 
On the contrary, the solver TFOCS \cite{convexTemplates} with implementation \url{http://cvxr.com/tfocs/} has been proposed for the solution of signal reconstruction problems without requiring
conditions on matrices $A$ and $W$ neither a matrix inversion at every iteration.
For the above reasons, in this section we compare pdNCG only with TFOCS. We think that this is a fair comparison since both methods, pdNCG and TFOCS, are developed to solve general signal reconstruction
problems, rather than focusing on few cases.



\subsection{Equivalent Problems}\label{subset:equiv}
Algorithms pdNCG and TFOCS implement different problems. In particular pdNCG solves problem \eqref{prob2},
while TFOCS solves the dual problem of 
\begin{equation}
  \begin{array}{lll}\label{prob4}
        & \displaystyle\min_{x\in\mathbb{R}^{n}} & \|W^*x\|_1 + \frac{\mu_{T}}{2}\|x - x^0\| \\
        &\mbox{subject to:}& \|Ax- b \|_2 \le \epsilon, \\
  \end{array}
\end{equation}
where $\epsilon$ is a positive constant and $\mu_{T}$ regulates the smoothing of the dual objective function. 
The two problems \eqref{prob2} and \eqref{prob4} are approximately equivalent if the smoothing terms $\mu$ and $\mu_{T}$ are very small and $c$ in \eqref{prob2} 
is defined as $c:= 2/\lambda_T$, where $\lambda_T$ is the optimal Lagrange multiplier of \eqref{prob4}. The exact optimal Lagrange multiplier $\lambda_T$ is not known a-priori.
However it can be calculated by solving to high accuracy the dual problem of \eqref{prob4} with TFOCS. Unforunately, the majority of the experiments
that we perform are large scale and TFOCS converges slowly for $\mu_{T}\approx 0$. For this reason, we first solve \eqref{prob4} using TFOCS
with a moderate $\mu_T$, in order to obtain an approximate optimal Lagrange multiplier $\lambda_T$ in reasonable CPU time. 
Then we set $c:= 2\gamma/\lambda_T$, where $\gamma$ is a small positive constant which is calculated experimentally 
such that the two solvers produce similar solution. 
Moreover, 
$\epsilon:= \|b - \tilde{b}\|_2$, where $\tilde{b}$ are the noiseless sampled data. If $\tilde{b}$ is not available, then $\epsilon$ is set such that
a visually pleasant solution is obtained.
The smoothing parameter $\mu_T$  of TFOCS is set such that the obtained solution, denoted by $x_T$, has moderately small relative error
$\|x_{T}-\tilde{x}\|_2/\|\tilde{x}\|_2$, 
where $\tilde{x}$ is
the known optimal noiseless solution.
Again if $\tilde{x}$ is not available, $\mu_T$ is set such that a visually pleasant reconstruction is obtained.
The smoothing parameter $\mu$ of pdNCG is set such that $\|x_{pd}-x_T\|_2/\|x_{T}\|_2$ is small, where $x_{pd}$
is the approximate optimal solution obtained by pdNCG.
For all experiments that were performed the relative error between 
the solution of TFOCS and pdNCG is of order $1.0e$-$2$.

\subsection{Termination Criteria, Parameter Tuning and Hardware}
The version $1.3.1$ of TFOCS has been used. The termination criterion of TFOCS is by default the relative step-length. The tolerance for this criterion is set to the default 
value, except in cases that certain suggestions are made in TFOCS software package or the corresponding paper \cite{convexTemplates}. 
The default Auslender \& Teboulle's single-projection method
is used as a solver for TFOCS. Moreover, as suggested by the authors of TFOCS, appropriate scaling is performed on matrices $A$ and $W$,
such that they have approximately the same Euclidean norms. All other parameters are set to their default values, except in cases
that specific suggestions are made by the authors. Generally, regarding tuning of TFOCS, substantial effort has been made in 
guaranteeing that problems are not over-solved.

Regarding pdNCG, the solver is employed until an approximately optimal solution is obtained, denoted by $x_{pd}$,
such that $\|x_{pd}-\tilde{x}\|_2/\|\tilde{x}\|_2 \le \|x_{T}-\tilde{x}\|_2/\|\tilde{x}\|_2$.
Parameter $\eta$ in \eqref{bd59} is set to $1.0e$-$1$, the maximum number of backtracking line-search iterations
is fixed to $10$. Moreover, the backtracking line-search parameters $\tau_1$ and $\tau_2$ in step $4$ of pdNCG (Fig. \ref{fig:2}) are set to $9.0e$-$1$ and $1.0e$-$3$,
respectively. For iTV the preconditioner is a five-diagonal matrix, hence systems with it are solved exactly. For general $\ell_1$-analysis, systems are solved approximately
with the preconditioner
using $15$ CG iterations. Finally, the constant $\rho$ of the preconditioner in \eqref{bd10} is set 
to $5.0e$-$1$.

Both solvers are MATLAB implementations and all experiments are run on a MacBook Air running OS X $10.9.2$ (13C64) with 2 GHz Intel Core i7 processor
using MATLAB R2012a. 

\subsection{$\ell_1$-analysis}
In this subsection we compare TFOCS and pdNCG on the recovery of radio-frequency radar tones. This problem has been first demonstrated 
in subsection $6.5$ of \cite{convexTemplates}. We describe again the setting of the experiment. 
The signal to be reconstructed consists of two radio-frequency radar tones which overlap in time.  The amplitude of the tones differs
by $60$ dB. The carrier frequencies and phases are chosen uniformly at random. Moreover, noise is added such that the larger tone has SNR (signal-to-noise-ratio)
$60$ dB and the smaller tone has SNR $2.1e$-$2$ dB.
The signal is sampled at $2^{15}$ points, which corresponds to 
Nyquist sampling rate for bandwidth $2.5$ GHz and time period approximately $6.5e$$\mplus$$3$ ns.
The reconstruction is modelled as a CS problem where the measurement matrix $A\in\mathbb{R}^{m\times n}$ 
is block-diagonal with $\pm 1$ for entries, $n=2^{15}$ and $m=2616$, i.e. subsampling ratio $m/n \approx 7.9e$-$2$. Moreover,
$W\in\mathbb{R}^{n\times l}$ is a Gabor frame with $l=915456$. The results of the comparison are presented in Figure \ref{figl1}. 
Observe that both solvers recovered a solution of similar accuracy but pdNCG was $1.3$ times faster. It is important to mention that the problems 
were not over-solved. TFOCS was tuned as suggested by its authors in a similar experiment which is shown in Subsection $6.5$ of \cite{convexTemplates}. 
This resulted in termination of TFOCS after $159$ iterations, which is considered as few for a first-order method.
\begin{figure}%
\centering
	\begin{subfigure}[b]{0.48\textwidth}
		\includegraphics[width=\textwidth]{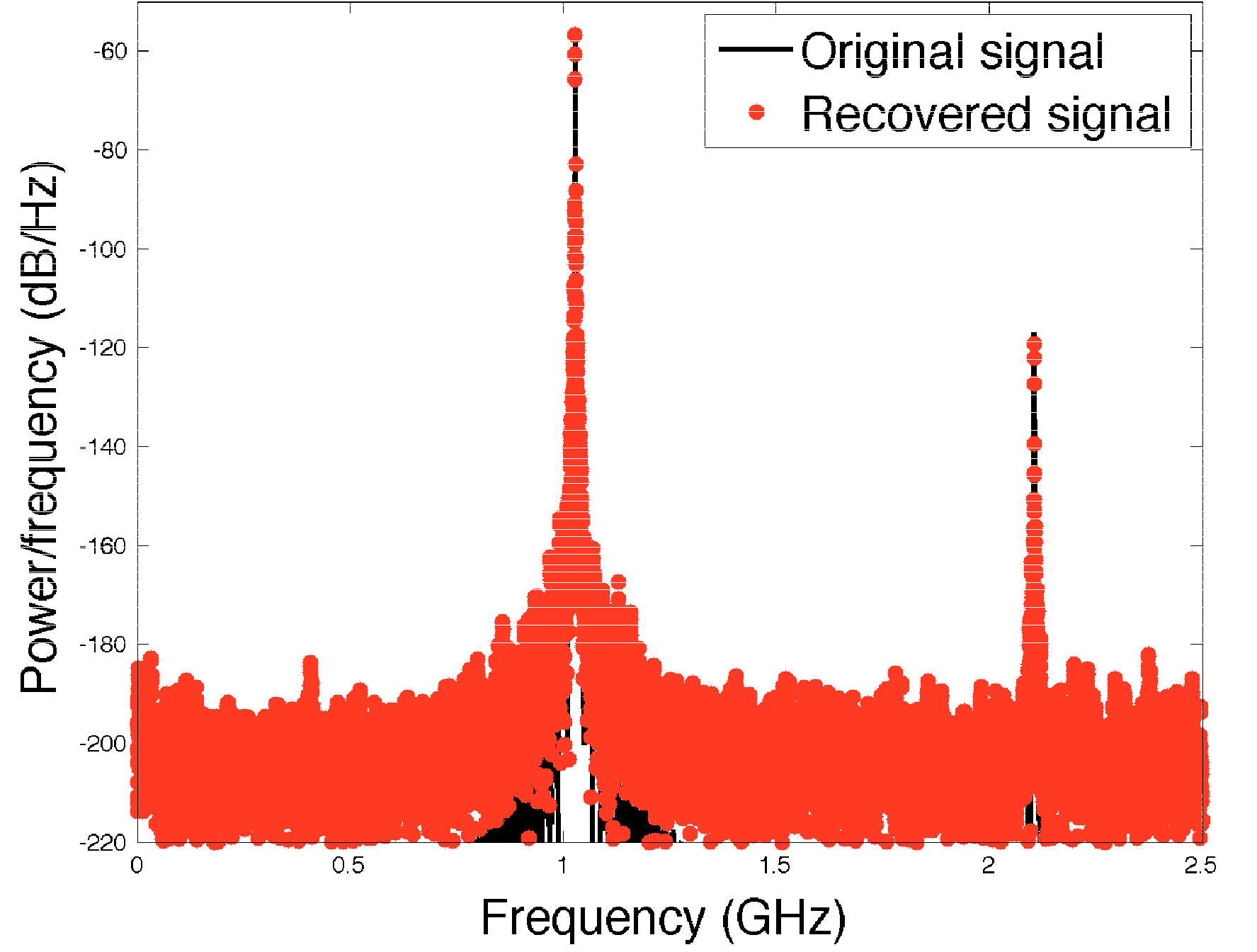}		
		\caption{TFOCS}
		\label{fig4}%
         \end{subfigure}
         \quad
	\begin{subfigure}[b]{0.48\textwidth}
		\includegraphics[width=\textwidth]{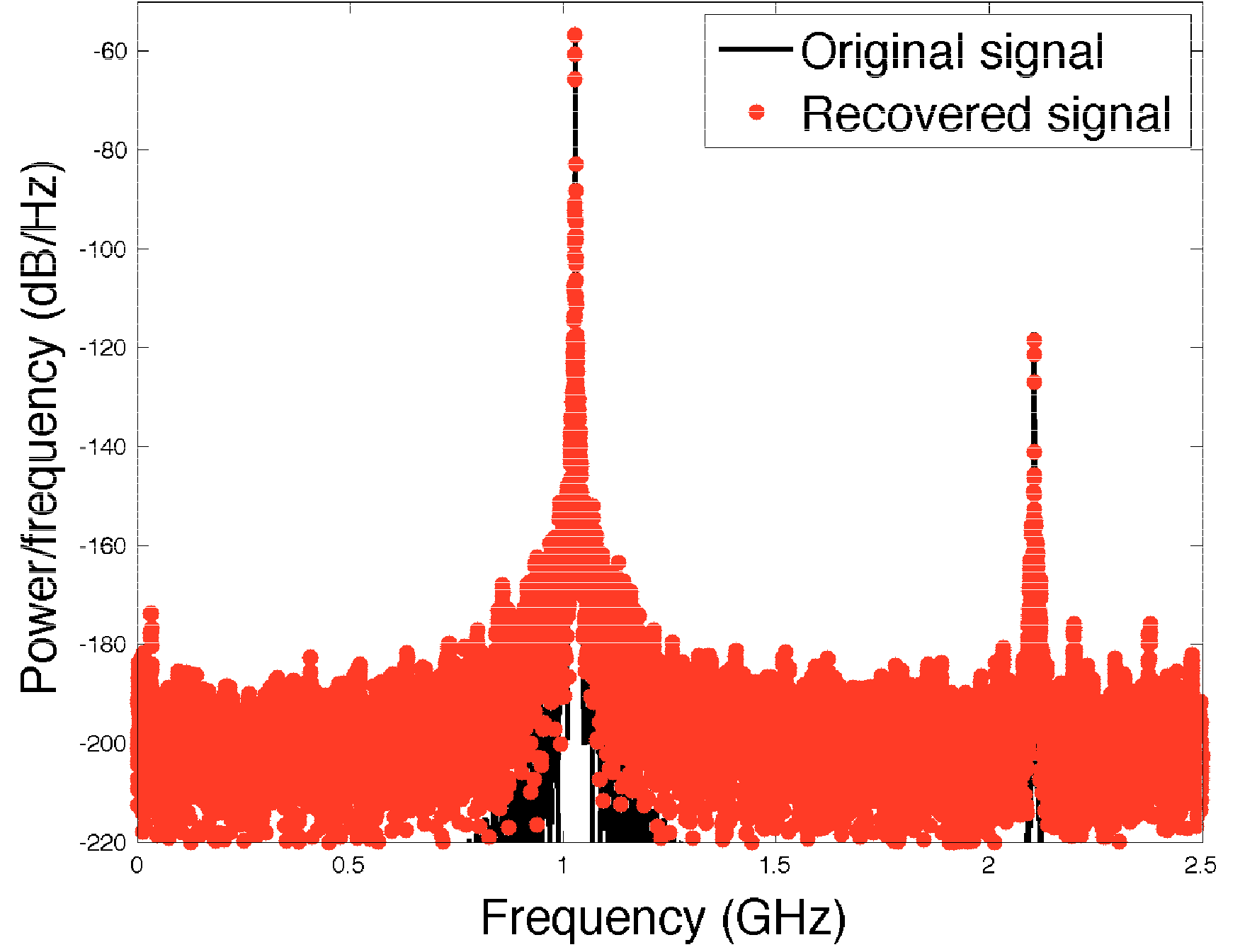}	
		\caption{pdNCG}
		\label{fig5}%
         \end{subfigure}
	\caption{Reconstruction of two radio-frequency radar tones by TFOCS and pdNCG using $\ell_1$-analysis. In Figure \ref{fig4} the reconstructed signal by TFOCS is shown. The signal has relative error $7.27e$-$4$
	and it required $620$ seconds CPU time to be reconstructed. In Figure \ref{fig5}
	the reconstructed signal by pdNCG is shown. The signal has relative error $5.87e$-$4$ and it required $480$ seconds CPU time to be reconstructed.
	}
	\label{figl1}%
\end{figure}

\subsection{Isotropic Total-Variation}
In this subsection we compare TFOCS and pdNCG on a synthetic image reconstruction problem which is modelled using iTV.
The image to be reconstructed is the well-known Shepp-Logan phantom image of size $256\times 256$ pixels shown in Figure \ref{figPhantomclean}. 
Noise is added to the original image such that it has pick-signal-to-noise-ratio (PSNR) $22.2$ dB, where
\begin{equation}\label{eq130}
\mbox{PSNR}(x) :=  20 \log_{10}\left(\frac{\sqrt{n_hn_v}}{\|x - \tilde{x}\|_F}\right),
\end{equation}
$x\in\mathbb{R}^{n_h\times n_v}$ represents an image with pixels values between zero and one, $n_h$ is the number the horizontal pixels, $n_v$ is the number of vertical pixels,
$\tilde{x}$ is an approximate optimal solution and
and $\|\cdot\|_F$ is the Frobenius norm.   
The noisy image is shown in Figure \ref{figPhantomnoise}. Then $25\%$ of all linearly projected noisy pixels are chosen uniformly at random, which consist the noisy sampled data $b$.
The projection matrix $A\in\mathbb{R}^{m\times n}$ 
is a partial $2D$ discrete cosine transform (DCT) with $n=256^2$ and $m\approx n/4$. The results of the comparison are shown in Figures \ref{fig6} and \ref{fig7}.
Both solvers reconstructed an image of similar quality, while pdNCG was $3.7$ times faster.
\begin{figure}%
\centering
	\begin{subfigure}[b]{0.48\textwidth}
		\includegraphics[width=\textwidth]{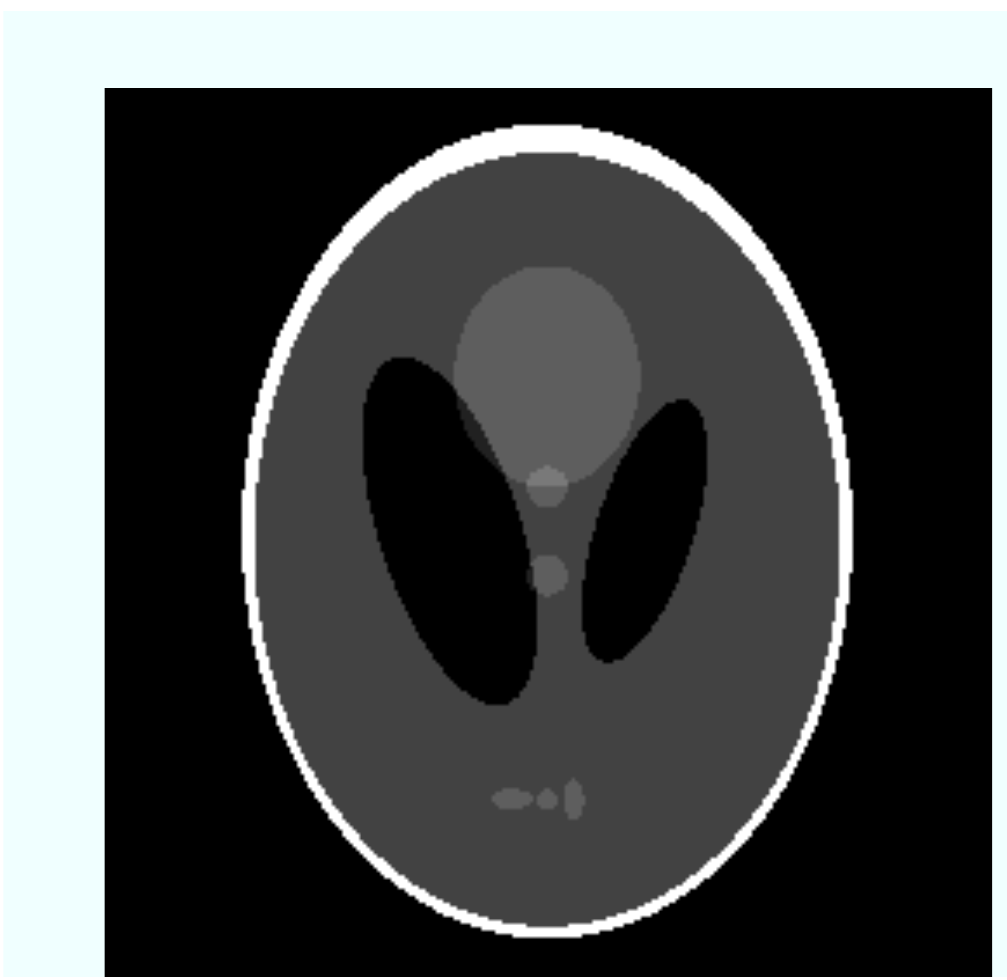}		
		\caption{Shepp-Logan phantom image}
		\label{figPhantomclean}%
         \end{subfigure}
         \quad
	\begin{subfigure}[b]{0.48\textwidth}
		\includegraphics[width=\textwidth]{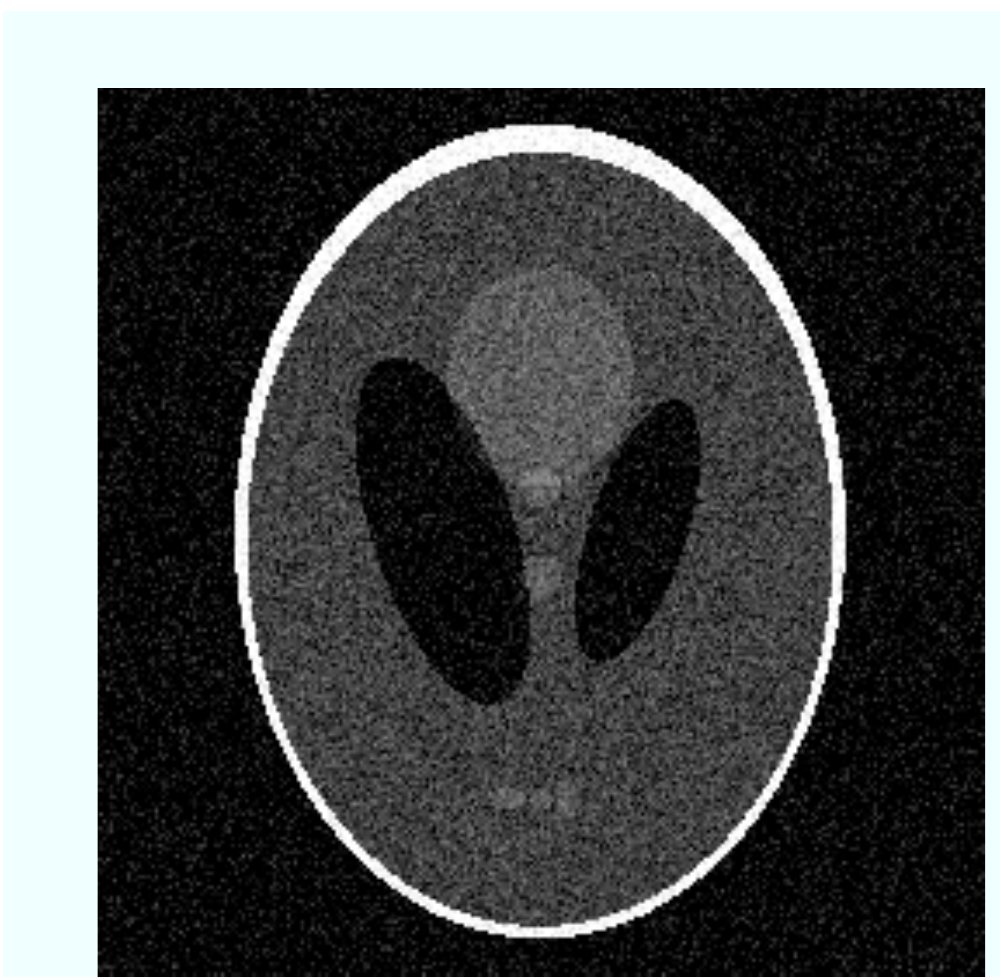}	
		\caption{Noisy version}
		\label{figPhantomnoise}%
         \end{subfigure}
	\begin{subfigure}[b]{0.48\textwidth}
		\includegraphics[width=\textwidth]{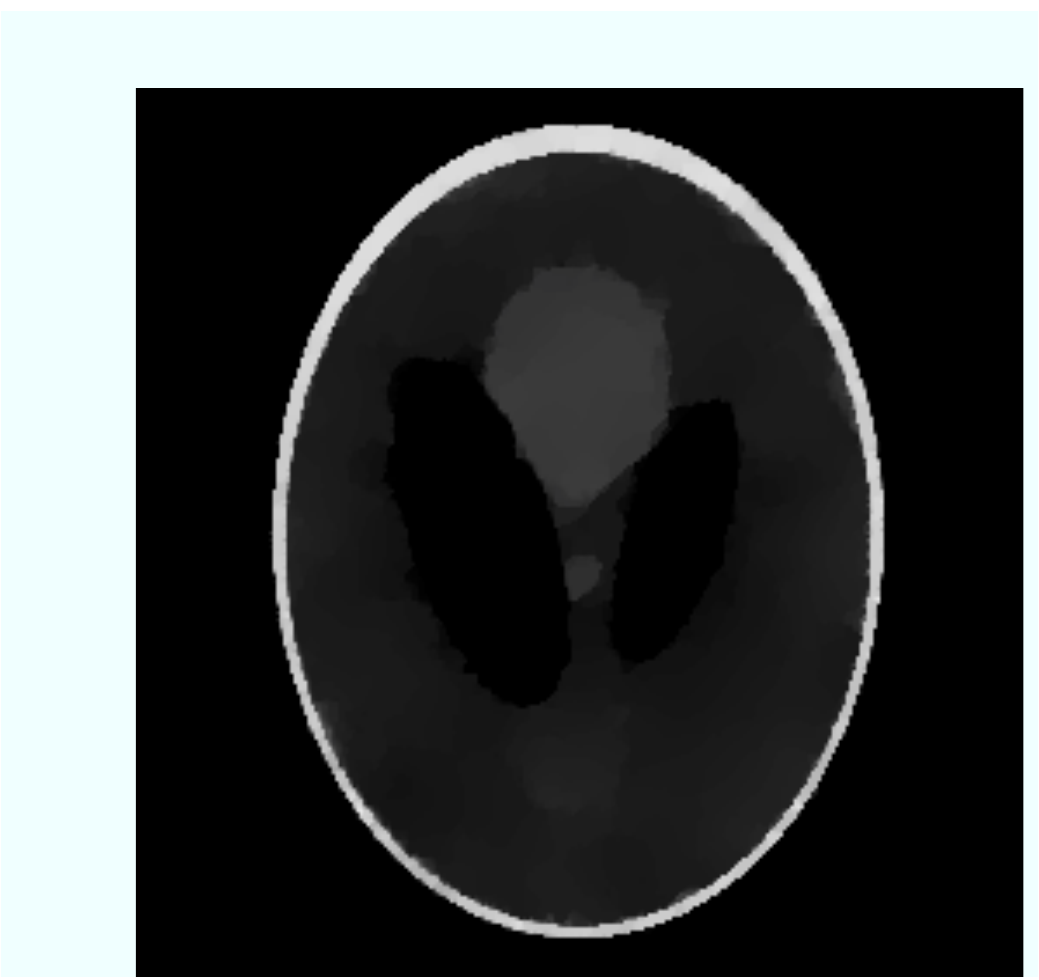}		
		\caption{Reconstructed by TFOCS}
		\label{fig6}%
         \end{subfigure}
         \quad
	\begin{subfigure}[b]{0.48\textwidth}
		\includegraphics[width=\textwidth]{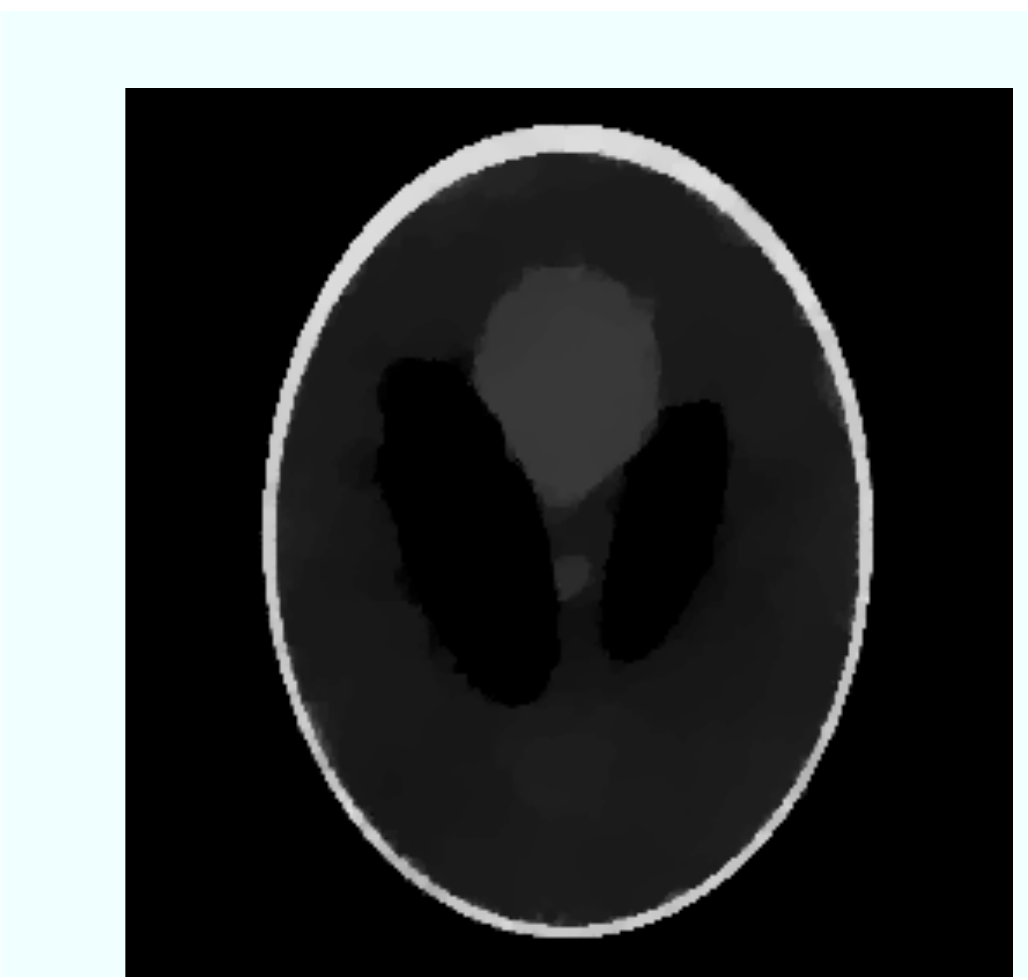}	
		\caption{Reconstructed by pdNCG}
		\label{fig7}%
         \end{subfigure}
	\caption{Figure \ref{figPhantomclean} is the noiseless Shepp-Logan phantom image, $256\times 256$ pixels.
	Figure \ref{figPhantomnoise} is the noisy version of the image with PSNR $22.2$ dB. 
	Figure \ref{fig6} is the reconstructed image by TFOCS using iTV. The image has PSNR $17.8$ dB
	and it required $60.7$ seconds CPU time to be reconstructed. Figure \ref{fig7}
	is the reconstructed image by pdNCG using iTV. The image has relative error $17.8$ dB and it required $16.4$ seconds CPU time to be reconstructed.
	}
	\label{figiTV}%
\end{figure}

\subsection{Single-Pixel Camera}
We now compare TFOCS with pdNCG on realistic image reconstruction problems where the data have been sampled
using a single-pixel camera \url{http://dsp.rice.edu/cscamera}. Briefly a single-pixel camera samples random linear projections of pixels
of an image, instead of directly sampling pixels. The problem set can be downloaded from \url{http://dsp.rice.edu/cscamera}.
In this set there are in total five sampled images, the dice, the ball, the mug the letter R and the logo. Each image has $64\times 64$ pixels.
The images are reconstructed using iTV. 
Unfortunately the optimal solutions are unknown for any requested subsampling level, additionally the level of noise is unknown. 
Hence the reconstructed images can only be compared 
by visual inspection.
For all four experiments $40\%$ of all linearly projected pixels are selected uniformly at random. 
The projection matrix $A\in\mathbb{R}^{m\times n}$, where $n=64^2$ and $m=0.4 n$, is a partial Walsh basis which takes values $0/1$ instead of $\pm 1$.
The reconstructed images by the solvers TFOCS and pdNCG are presented in Figure \ref{figcscamera}. Solver pdNCG was faster on four out of five problems.
On problems that pdNCG was faster it required on average $1.4$ times less CPU time. Although it would be possible to tune pdNCG such that it is faster on all problems, we
preferred to use its (simple) default tuning in order to avoid a biased comparison. 
\begin{figure}%
\centering
	\begin{subfigure}[b]{0.25\textwidth}
		\includegraphics[width=\textwidth]{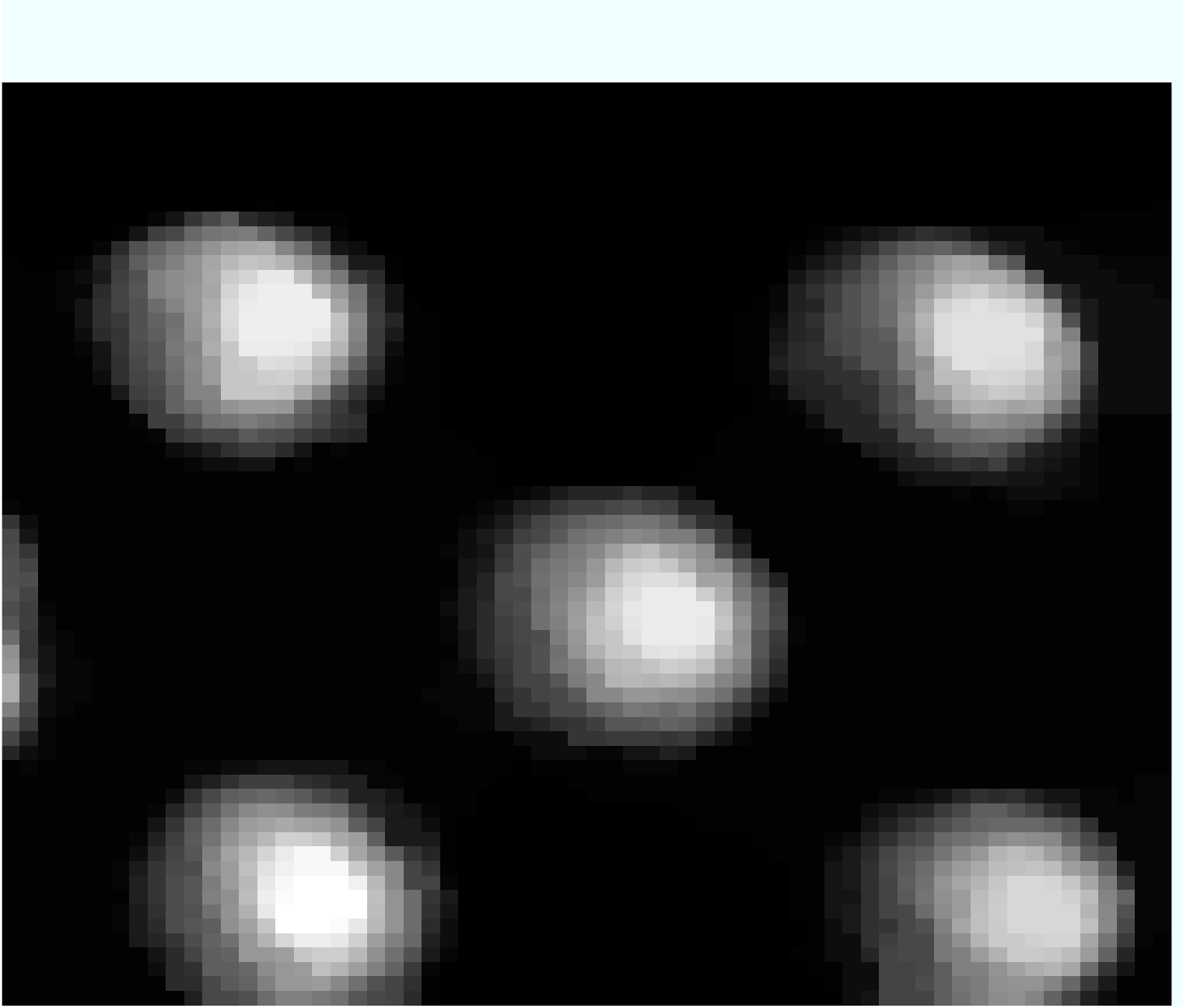}		
		\caption{TFOCS, $25$ sec.}
		\label{figcscamera_a}%
         \end{subfigure}
         \quad
	\begin{subfigure}[b]{0.25\textwidth}
		\includegraphics[width=\textwidth]{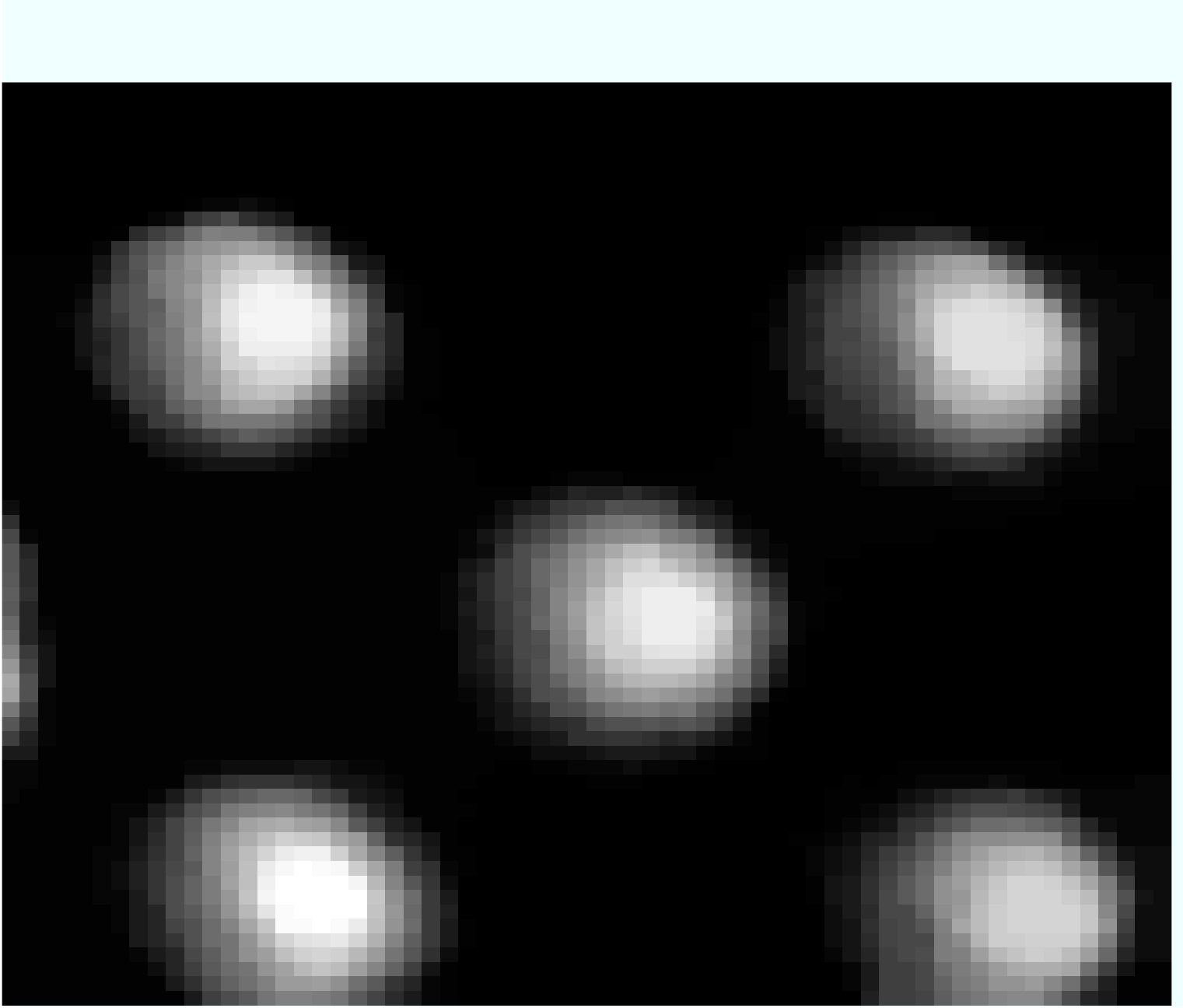}	
		\caption{pdNCG, $19$ sec.}
		\label{figcscamera_b}%
         \end{subfigure}
         \\
	\begin{subfigure}[b]{0.25\textwidth}
		\includegraphics[width=\textwidth]{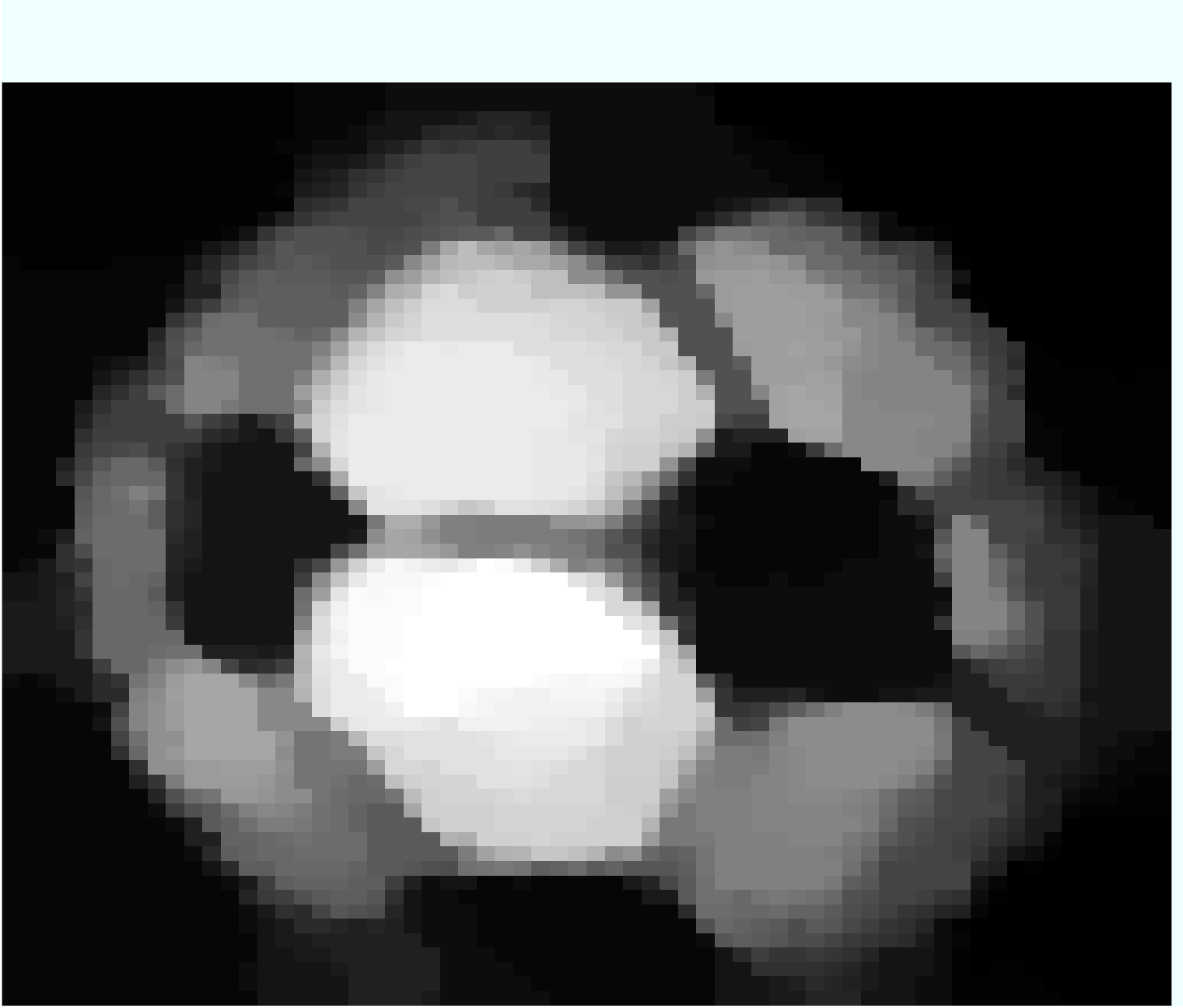}		
		\caption{TFOCS, $24$ sec.}
		\label{figcscamera_c}%
         \end{subfigure}
         \quad
	\begin{subfigure}[b]{0.25\textwidth}
		\includegraphics[width=\textwidth]{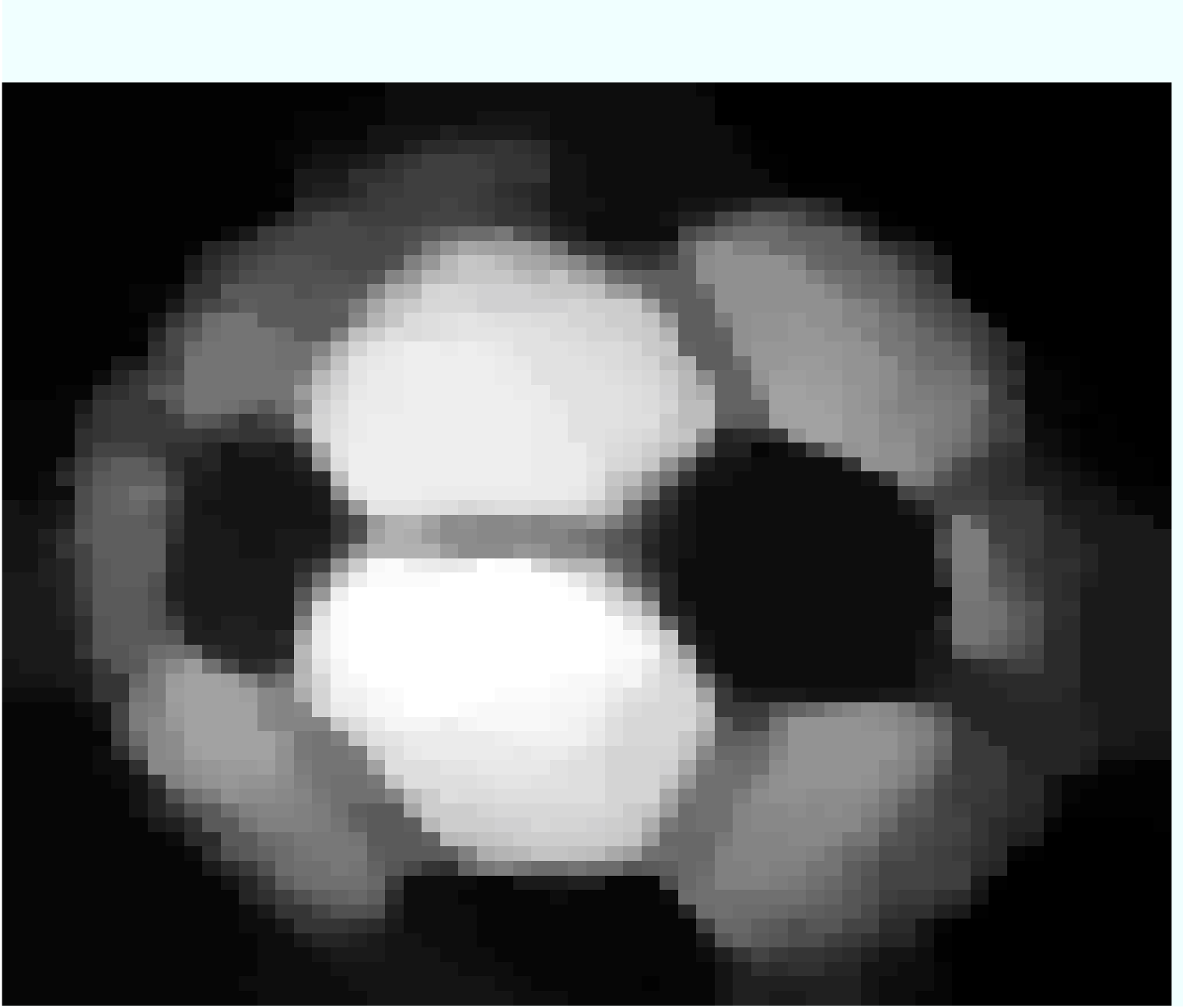}	
		\caption{pdNCG, $15$ sec.}
		\label{figcscamera_d}%
         \end{subfigure}
         \\
	\begin{subfigure}[b]{0.25\textwidth}
		\includegraphics[width=\textwidth]{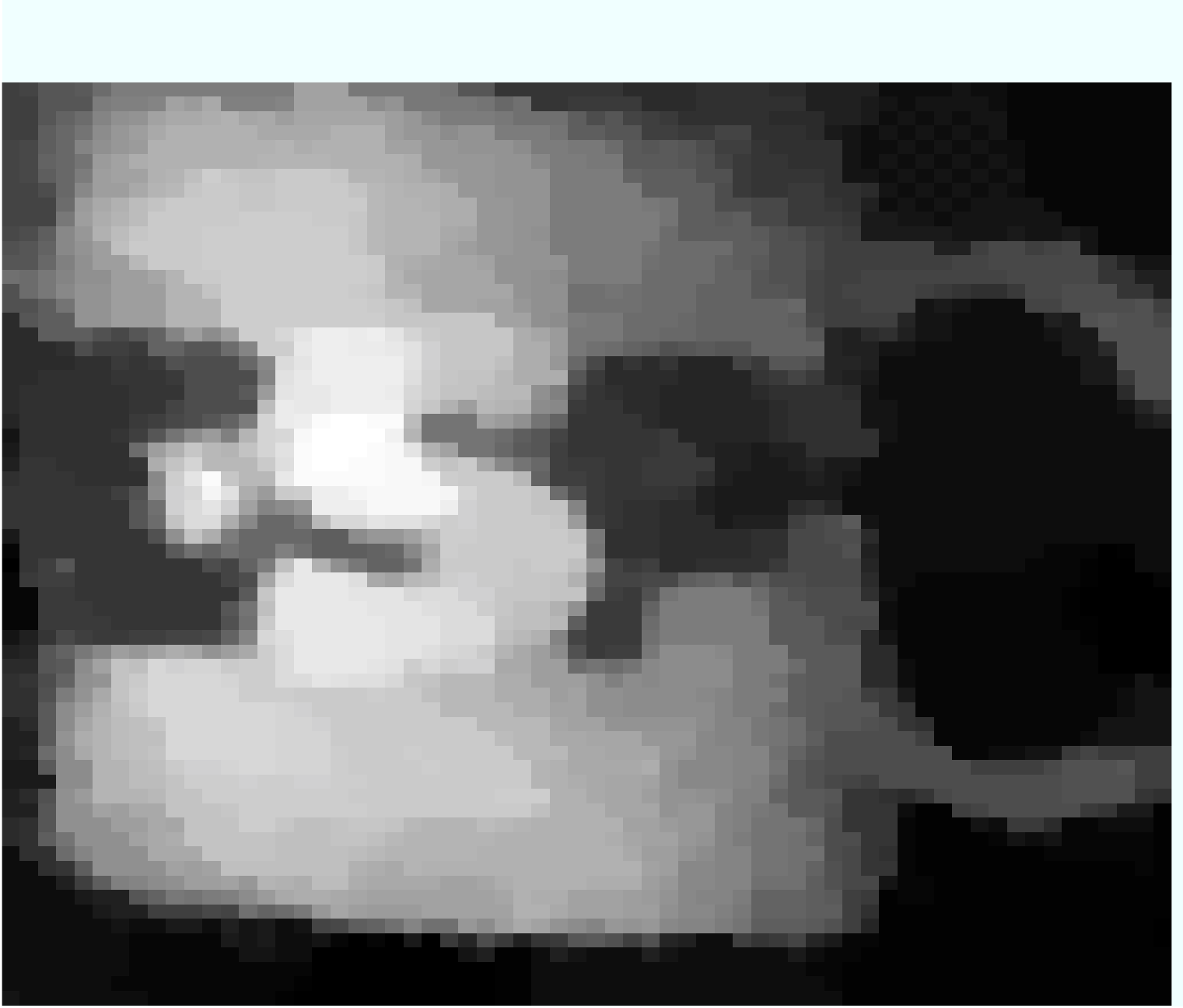}	
		\caption{TFOCS, $37$ sec.}
		\label{figcscamera_e}%
         \end{subfigure}
         \quad
	\begin{subfigure}[b]{0.25\textwidth}
		\includegraphics[width=\textwidth]{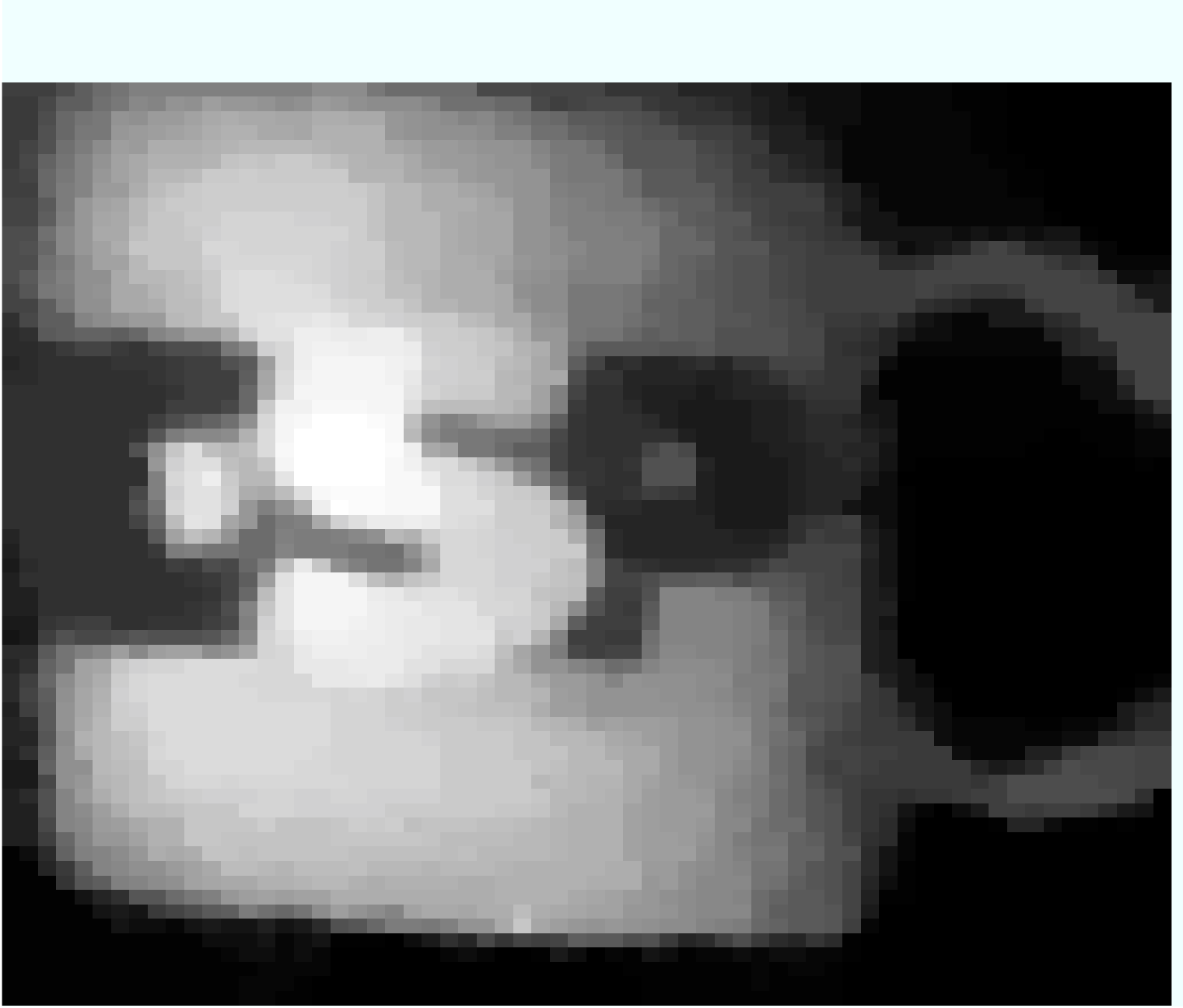}	
		\caption{pdNCG, $15$ sec.}
		\label{figcscamera_f}%
         \end{subfigure}
         \\
	\begin{subfigure}[b]{0.25\textwidth}
		\includegraphics[width=\textwidth]{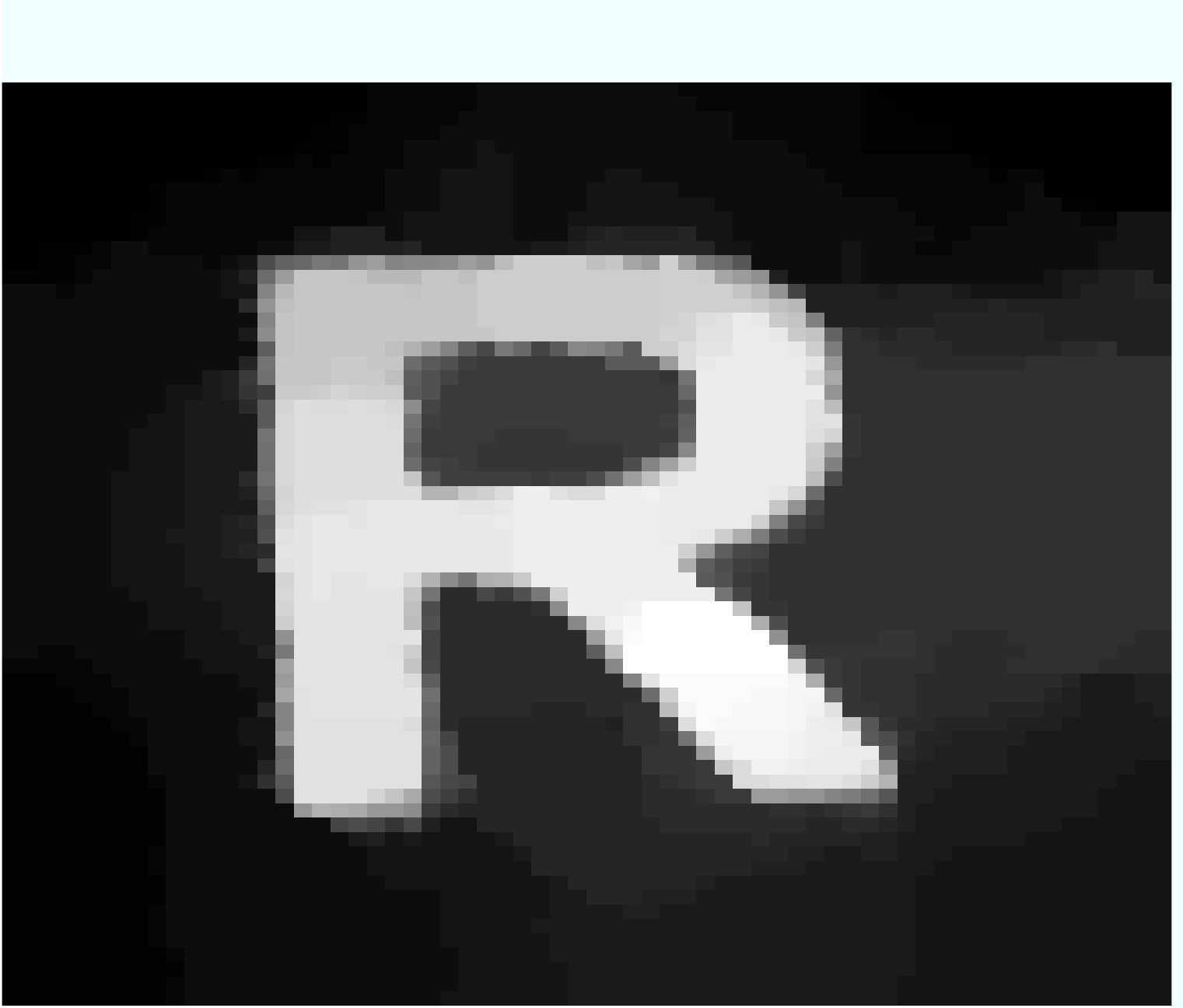}	
		\caption{TFOCS, $24$ sec.}
		\label{figcscamera_g}%
         \end{subfigure}
         \quad
	\begin{subfigure}[b]{0.25\textwidth}
		\includegraphics[width=\textwidth]{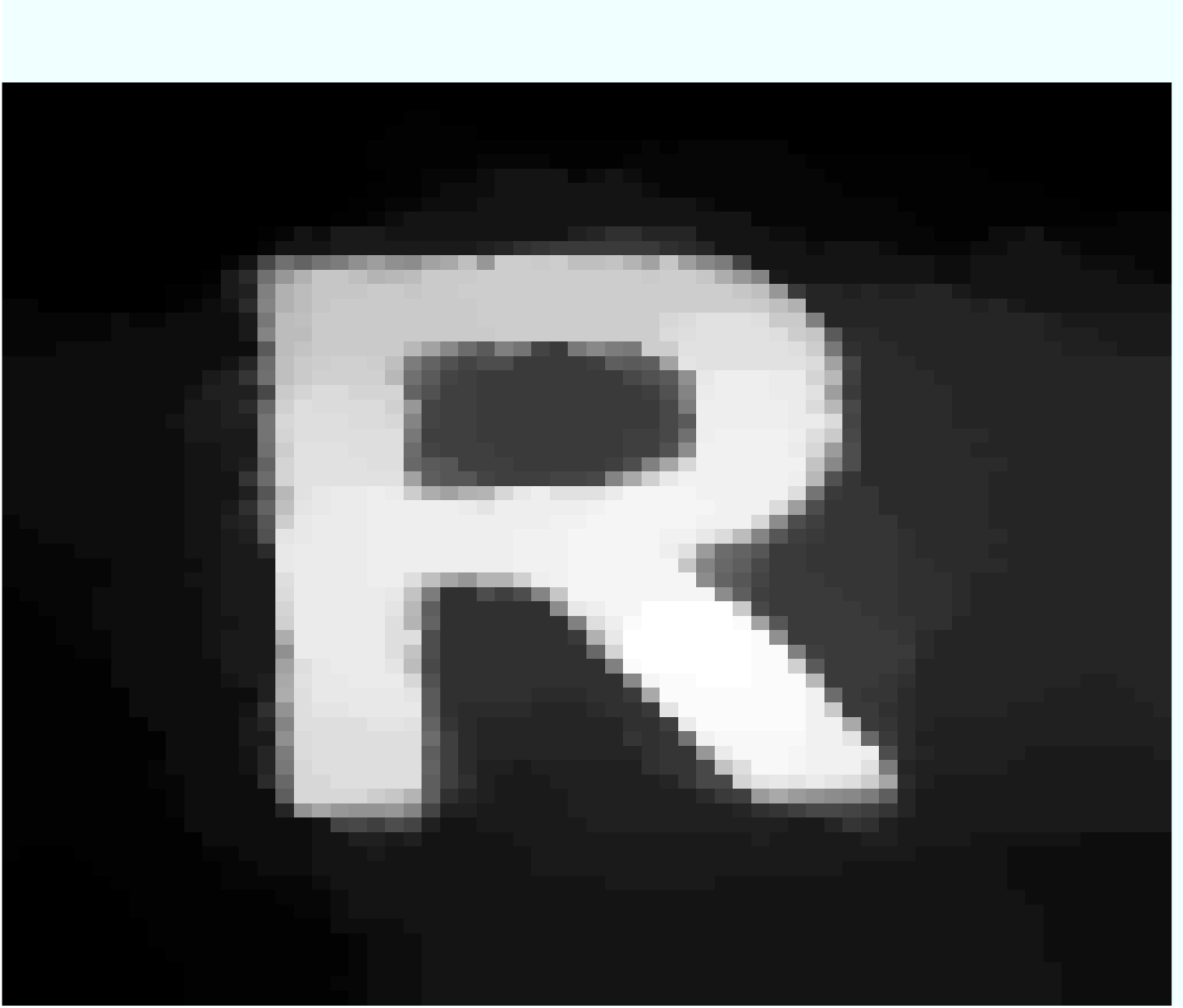}	
		\caption{pdNCG, $32$ sec.}
		\label{figcscamera_h}%
         \end{subfigure}
         \\
	\begin{subfigure}[b]{0.25\textwidth}
		\includegraphics[width=\textwidth]{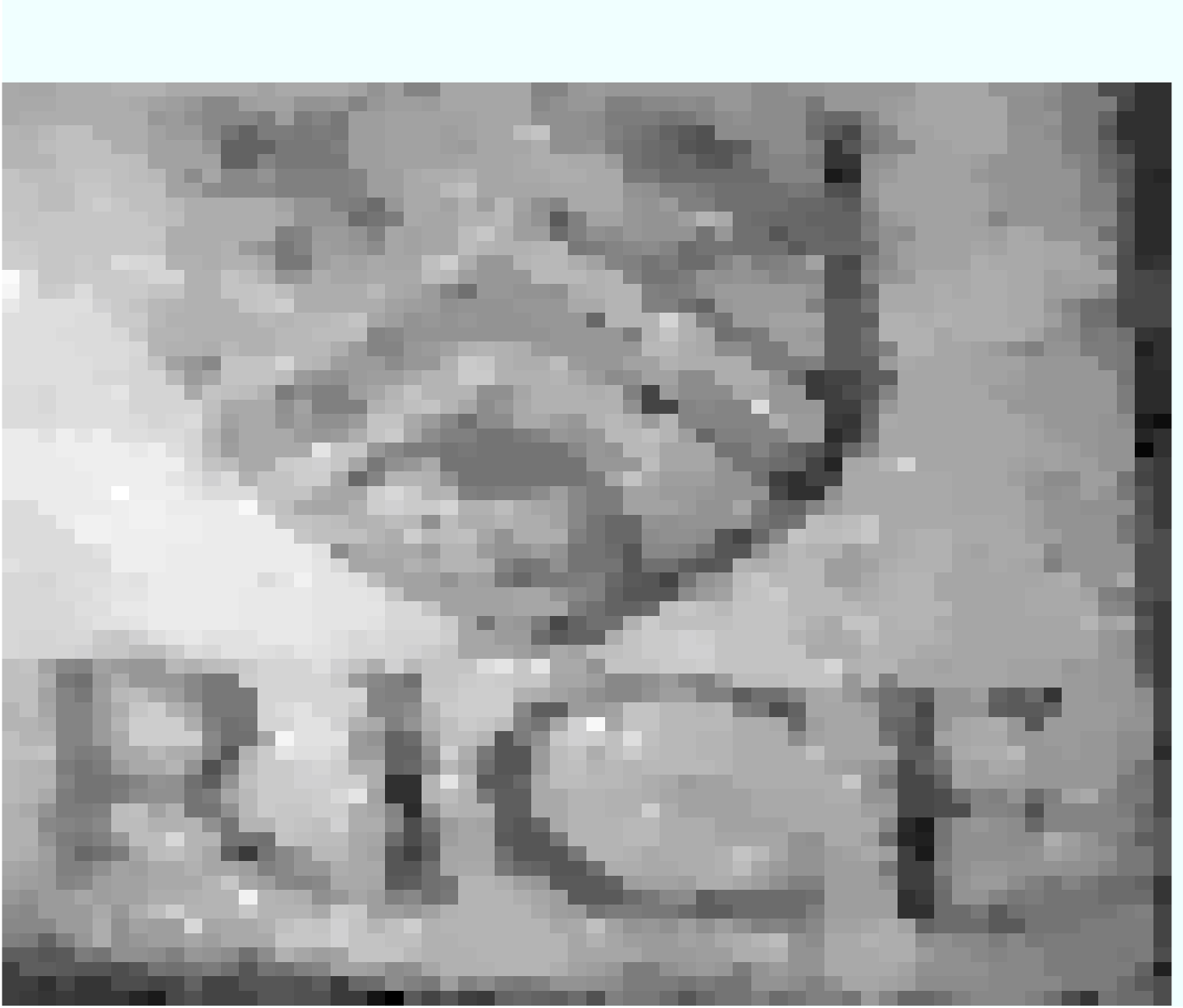}	
		\caption{TFOCS, $49$ sec.}
		\label{figcscamera_j}%
         \end{subfigure}
         \quad
	\begin{subfigure}[b]{0.25\textwidth}
		\includegraphics[width=\textwidth]{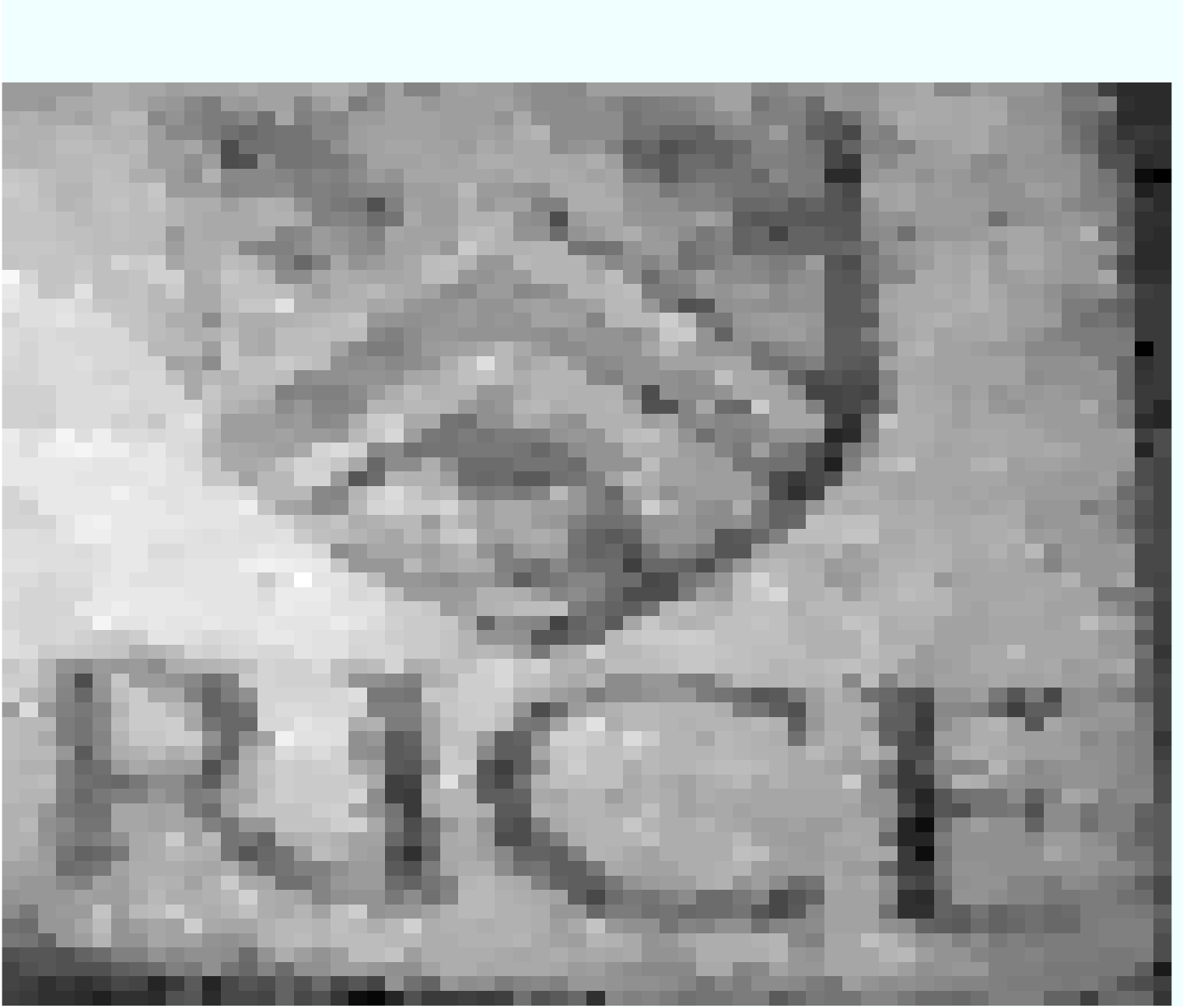}	
		\caption{pdNCG, $33$ sec.}
		\label{figcscamera_k}%
         \end{subfigure}
	\caption{Experiment on realistic image reconstruction where the samples are acquired using a single-pixel camera.
	The subcaptions of the figures show the required seconds of CPU time for the image to be reconstructed for each solver.
	}
	\label{figcscamera}%
\end{figure}

\section{Conclusions}\label{sec:concl}
Recently there has been great interest in the development of optimization methods for the solution of compressed sensing 
problems with coherent and redundant dictionaries.
The methods that have been developed so far are mainly first-order methods.
This is because first-order methods have inexpensive iteration cost and frequently 
offer fast initial progress in the optimization process. On the contrary, second-order methods are considered to be rather expensive. 
The reason is that often access to second-order information requires the solution of linear systems. In this paper we 
develop a second-order method, a primal-dual Newton Preconditioned Conjugate Gradients. We show
that \textit{approximate} solution of linear systems which arise is sufficient to speed up an iterative method and additionally
make it more robust. Moreover, we show that for compressed sensing problems inexpensive preconditioners can be designed that
speed up even further the approximate solution of linear systems. Extensive numerical experiments are presented which verify 
our arguments. In the theoretical front we prove convergence of pdNCG and local super-linear rate of convergence.

\bibliographystyle{plain}
\bibliography{KFandJG.bib}

\begin{thebibliography}{10}

\bibitem{csalsa}
M.~V. Afonso, J.~M. Bioucas-Dias, and M.~A.~T. Figueiredo.
\newblock An augmented {L}agrangian approach to the constrained optimization
  formulation of image inverse problems.
\newblock {\em IEEE Transactions on Image Processing}, 20(3):681--695, 2011.

\bibitem{convexTemplates}
S.~R. Becker, E.~J. Cand\'{e}s, and M.~C. Grant.
\newblock Templates for convex cone problems with applications to sparse signal
  recovery.
\newblock {\em Mathematical Programming Computation}, 3(3):165--218, 2011.

\bibitem{apps3}
E.~J. Cand\'{e}s and D.~L. Donoho.
\newblock New tight frames of curvelets and optimal representations of objects
  with piecewise $c^2$ singularities.
\newblock {\em Comm. Pure Appl. Math.}, 57:219--266, 2004.

\bibitem{l1analysis}
E.~J. Cand\'{e}s, Y.~C. Eldar, and D.~Needell.
\newblock Compressed sensing with coherent and redundant dictionaries.
\newblock {\em Applied and Computational Harmonic Analysis}, 31(1):59--73,
  2011.

\bibitem{ctnewtonold}
R.~H. Chan, T.~F. Chan, and H.~M. Zhou.
\newblock Advanced signal processing algorithms.
\newblock {\em in Proceedings of the International Society of Photo-Optical
  Instrumentation Engineers, F. T. Luk, ed., {SPIE}}, pages 314--325, 1995.

\bibitem{ctpdnewton}
T.~F. Chan, G.~H. Golub, and P.~Mulet.
\newblock A nonlinear primal-dual method for total variation-based image
  restoration.
\newblock {\em SIAM J. Sci. Comput.}, 20(6):1964--1977, 1999.

\bibitem{mybib:DES}
R.~S. Dembo, S.~C. Eisenstat, and T.~Steihaug.
\newblock Inexact {N}ewton methods.
\newblock {\em SIAM Journal on Numerical Analysis}, 19:400--408, 1982.

\bibitem{pdhg}
J.~E. Esser.
\newblock {\em Primal Dual Algorithms for Convex Models and Applications to
  Image Restoration, Registration and Nonlocal Inpainting}.
\newblock PhD thesis, University of California, 2010.

\bibitem{l1regSCfg}
K.~Fountoulakis and J.~Gondzio.
\newblock A second-order method for strongly convex $\ell_1$-regularization
  problems.
\newblock {\em Technical Report ERGO 14-005}, 2014.

\bibitem{Hartley2004}
R.~I. Hartley and A.~Zisserman.
\newblock {\em Multiple View Geometry in Computer Vision}.
\newblock Cambridge University Press, ISBN: 0521540518, second edition, 2004.

\bibitem{gista}
I.~Loris and C.~Verhoeven.
\newblock On a generalization of the iterative soft-thresholding algorithm for
  the case of non-separable penalty.
\newblock {\em Inverse Problems}, 27(12):1--15, 2011.

\bibitem{apps4}
S.~Mallat.
\newblock A wavelet tour of signal processing, second ed.
\newblock {\em Academic Press, London}, 1999.

\bibitem{tvrobust}
D.~Needell and R.~Ward.
\newblock Stable image reconstruction using total variation minimization.
\newblock {\em SIAM J. Imaging Sciences}, 6(2):1035--1058, 2013.

\bibitem{mybib:NocedalWright}
J.~Nocedal and S.~J. Wright.
\newblock {\em Numerical Optimization}.
\newblock Springer, New York, 2006.

\bibitem{fadili}
S.~Vaiter, G.~Peyr\'{e}, C.~Dossal, and J.~Fadili.
\newblock Robust sparse analysis regularization.
\newblock {\em IEEE Trans. Inf. Theory}, 59(4):2001--2016, 2013.

\bibitem{approxprec}
C.~R. Vogel and M.~E. Oman.
\newblock Fast, robust total variation-based reconstruction of noisy, blurred
  images.
\newblock {\em Image Processing, IEEE Transactions on}, 7(6):813--824, 1998.

\end{thebibliography}

\end{document}